\documentclass[10pt,letter]{article}
\usepackage{amsmath,amssymb} 
\usepackage{float}
\usepackage{graphicx}
\usepackage{enumerate}
\usepackage{amsmath,amssymb}
\usepackage[english]{babel}
\usepackage{epsfig}
\usepackage{pstricks,pst-node,pst-coil,pst-plot,pst-text}
\usepackage{ifpdf}
\usepackage{enumitem}
\usepackage{subfigure}
\usepackage{caption}
\usepackage{amsthm}
\usepackage{ifpdf}
\newtheorem{theorem}{Theorem}

\newtheorem{definition}{Definition}
\newtheorem{example}{Example}
\newtheorem{lemma}{Lemma}
\newtheorem{proposition}{Proposition}

\newtheorem{observation}{Observation}

\def\g{\gamma}

\usepackage{graphicx,color}

\usepackage{graphicx,color}
\begin{document}

\title{Directed Cycle Double Cover Conjecture: Fork Graphs}
%
%
\author{Andrea~Jim\'enez
\thanks{Instituto de Matem\'atica e Estat\'istica,  Universidade de S\~ao Paulo.
\texttt{ajimenez@ime.usp.br}. Supported by CNPq (Proc.~477203/2012-4) and FAPESP (Proc.~2011/19978-5).}
\and Martin~Loebl
\thanks{Department of Applied Mathematics, Charles University.
\texttt{loebl@kam.mff.cuni.cz}. Partially supported by the 
Czech Science Foundation GACR under the contract number P202-12-G061, CE-ITI.}}

\maketitle              

\begin{abstract}
We explore the well-known Jaeger's directed cycle double cover conjecture
which is equivalent to the assertion that every cubic bridgeless graph has 
an embedding on a closed orientable surface with no dual loop. 
We associate each cubic graph $G$ 
with a novel object $H$ that we call a \emph{hexagon graph}; perfect 
matchings of $H$ describe all embeddings of $G$ on closed orientable surfaces. 
The study of hexagon graphs leads us to define a new class of graphs that we call 
{\em lean fork-graphs}. Fork graphs are cubic bridgeless graphs obtained from a triangle by 
sequentially connecting fork-type graphs and performing Y${-}\Delta$, $\Delta{-}$Y transformations;
lean fork-graphs are fork graphs fulfilling a connectivity property.   
We prove that Jaeger's conjecture holds for the class of lean fork-graphs.
The class of lean fork-graphs is rich; namely, for each cubic bridgeless
graph $G$ there is a lean fork-graph containing a subdivision of $G$
as an induced subgraph.
Our results establish for the first time, to the best of our knowledge, 
the validity of Jaeger's conjecture in a broad inductively defined class of graphs. 
\end{abstract}

\section{Introduction}\label{sec:intro} 

One of the most challenging open problems in graph theory is the cycle double cover 
conjecture which was independently posed  by Szekeres~\cite{BAZ:4875124} and Seymour~\cite{MR538060} 
in the seventies. 
It states that every bridgeless graph has a cycle double cover, 
that is,~a system $\mathcal{C}$ of cycles such that each edge of the graph belongs to exactly two cycles of $\mathcal{C}$. 
Extensive attempts to prove the cycle double cover conjecture have led to many interesting 
concepts and conjectures. In particular, some of the stronger versions of the cycle double cover 
conjecture are related to embeddings of graphs on a surface. 
The Jaeger's directed cycle double cover conjecture~\cite{Jaeger19851}
states that every cubic graph with no bridge has a cycle double cover $\mathcal{C}$ to which 
one can prescribe orientations in such a way that the orientations of each edge of the graph 
induced by the prescribed orientations of the cycles are opposite. 
Jaeger's conjecture is equivalent to the statement that every cubic bridgeless graph has an 
 embedding on a closed orientable surface with no dual loop. 

In our previous work~\cite{jkl}, we took a new 
approach to Jaeger's conjecture, see Proposition~\ref{prop:dualloops}, 
motivated by the notion of critical embeddings. Critical embeddings are 
used extensively as a discrete tool towards mathematical understanding of criticality 
of basic statistical physics models of Ising and dimer, and conformal quantum field theory 
of free fermions~\cite{c,k,m}. We formulated the existence of 
embeddings of cubic bridgeless graphs with no dual loops as the existence of special 
perfect matchings in a subclass of braces that we call hexagon graphs. 


\subsection*{Main Contribution} 
In the current work, we introduce new key notions of {\em safe reductions} 
and {\em cut obstacles}. 
The main results of this work are summarized in 
Theorems~\ref{thm.main1},~\ref{thm.main2}
and~\ref{thm:contsubofcub}.
We prove that the directed cycle double cover conjecture 
is valid for all lean fork-graphs. The  class of all lean fork-graphs is natural and rich. 
On the one hand
this class is inductively defined starting from a triangle by
sequentially adding ``ears''; ears are the so-called fork-type graphs. 
On the other hand for each cubic bridgeless graph $G$ it is 
possible to construct a lean fork-graph that 
contains a subdivision of $G$ as an induced subgraph.

\subsection*{Related Work on dcdc} 
Jaeger's directed cycle double cover conjecture trivially 
holds in the class of cubic bridgeless planar graphs.
However, little is known about its validity in other classes of graphs.
Indeed, our results establish for the first time, 
to the best of our knowledge, the validity of Jaeger's 
conjecture in a rich inductively defined class of graphs.

\subsection*{Related Work on cdc} 
Much more is known about the weaker cycle double cover conjecture.
Jaeger~\cite{Jaeger19851} proved that any minimal counterexample to the cycle double cover conjecture is a snark;
namely,~a connected cubic graph which cannot be properly edge-colored with three colors. 
The famous snark is the Petersen graph. 
Alon and Tarsi~\cite{AlonTarsi} conjectured that the edge set of every 
bridgeless cubic graph with $m$ edges has a cycle cover where the total sum of the length of the cycles is
at most $7m/5$, and Jamshy and Tarsi~\cite{JamshyTarsi} 
proved that this conjecture implies the cycle double cover conjecture.

Existence of a cycle double cover in the classes of 3-edge-colorable and 4-edge-connected cubic 
graphs have also been positively settled~\cite{Jaeger1979205}. 
The cycle double cover conjecture also holds for all cubic bridgeless graphs 
that do not contain a subdivision of the Petersen's graph~\cite{journals/dm/AlspachZ93} 
and for graphs which have Hamiltonian paths~\cite{Goddyn1989253}. 
The important connection with the theory of nowhere-zero-flows 
is exploited in~\cite{MR1395462}.

In the next section we present the main ideas on which our work is based and formally
establish our contributions.


\section{Main Ideas and Results}\label{sec:ideasandresults}
One natural way of starting the construction of a directed cycle double cover of a cubic bridgeless 
graph $G$ with vertex set $V$ and edge set $E$ is to select a vertex $v\in V$
and to \emph{wire} its incident edges $\{v,x\}$, $\{v,y\}$, $\{v,z\}$: 
this creates a directed triangle consisting of three new directed edges $(x,y),(y,z),(z,x)$. 
Once such a directed triangle is formed, the vertex $v$ and the edges $\{v,x\}, \{v,y\}, \{v,z\}$ are 
deleted from $G$, resulting in a new \emph{mixed} graph with vertex set $V{-}\{v\}$ and 
edge set $E{-}\{\{v,x\}, \{v,y\}, \{v,z\}\}$, together with a set $\{(x,y),(y,z),(z,x)\}$ 
of directed edges of the triangle. 
We can continue this procedure by sequentially selecting a vertex  
$u$ in $V{-}\{v\}$ and wiring its incident edges and arcs.
We note that for some pairs of the created directed edges, it is forbidden to belong to the same cycle 
of a directed cycle double cover. 
{\em If} we could continue this procedure until every edge is wired, we {\em might} be able to 
show the existence of a directed cycle double cover. 
But, this naive approach leads to the following crucial questions: what do mixed graphs look 
like, in the middle of the wiring  procedure? 
For which classes of graphs is it possible to 
apply the wiring procedure until we find a directed cycle double cover?  
What are \emph{obstacles} that hinder the continuation of the wiring procedure? 
The first question leads to the definition of {\em mixed graphs}, 
and to the novel concept of {\em safe reductions}. The last two questions to concepts of 
{\em fork-collections}, {\em fork-graphs}, and {\em cut-obstacles}.

A {\em mixed graph} is a 4-tuple $(V, E, A, R)$ where $V$ is the vertex set,
$E$ is the edge set, $A$ is the set of the directed edges (arcs), and $R$ is a subset of $A{\times} A$,
that is, a set of pairs of arcs. We require that in the graph
induced by $E$, each vertex has degree at least two 
and at most three, and that each vertex of degree two is the tail
of exactly one arc and the head of exactly one arc. 
Regarding the discussion in the previous paragraph, the set $R$
contains those pairs of arcs, which cannot be together in
the same directed cycle of the constructed directed cycle double cover.

\subsection*{Safe reductions} 
A {\em reduction} of a subset $S$ of the vertices of a mixed graph is defined 
naturally as wiring the edges and directed edges incident with $S$, and updating $R$. 
However, $R$ becomes complicated and actually our life would be easier if we could avoid it. 
It turns out that indeed updating $R$ is not necessary if we perform {\em safe reductions}.
\begin{figure}[h]
\centering
\subfigure[]
{
\ifpdf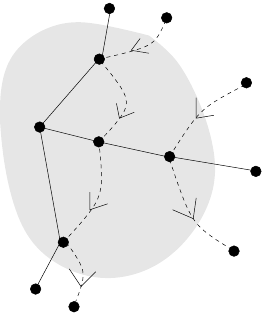\else\input{safe-red-mixedgraphs_00.ps_t}\fi
 \label{fig.srmg00}
}
\subfigure[]
{
\ifpdf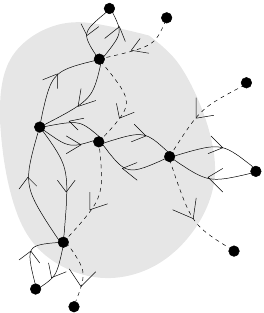\else\input{safe-red-mixedgraphs_01.ps_t}\fi
  \label{fig.srmg01}
}
\subfigure[]
{
\ifpdf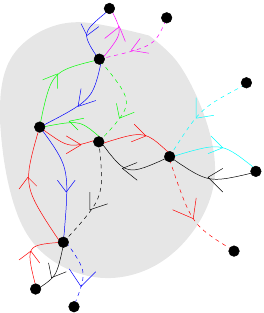\else\input{safe-red-mixedgraphs_02.ps_t}\fi
 \label{fig.srmg02}
}
\subfigure[]
{
\ifpdf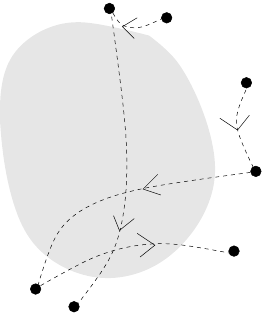\else\input{safe-red-mixedgraphs_03.ps_t}\fi
  \label{fig.bf3}
}
\caption{Safe reduction of a subset of vertices $S$:
(a) elements from $A$ are re\-presented by dotted lines,
(b) replace edges in $E$ by two arcs oppositely directed,
(c) partition of $A_S\cup A'$ into safe paths and cycles and
(d) resulting structure.}
\label{fig.safe-reduction-mixed-graphs}
\end{figure}

Let $(V, E, A, R)$ be a mixed graph and $S{\subset} V$. 
We say that a mixed graph is obtained from $(V, E, A, R)$ by a {\em safe reduction} of 
$S$ if it is constructed as follows. We 
replace each edge in $E$ incident to a vertex of $S$ by two 
arcs oppositely directed; let these new arcs form the set $A'$. 
Let $A_S$ be the subset of $A$ that contains all directed edges incident to a vertex of $S$. Then we partition 
$A_S{\cup} A'$ into safe directed cycles 
and safe directed paths with both end vertices in $V{-} S$. A cycle or a 
path is {\em safe} if it has at most one edge from $A_S$, 
and if it is not a 2-cycle composed of
only edges from $A'$. 
Finally we replace each 
chosen safe path by the directed edge between its end vertices
and then, we delete the vertices of~$S$. In Figure~\ref{fig.safe-reduction-mixed-graphs}, we show an example of a safe reduction.

\subsection*{Consecutive safe reductions} 
In order to construct directed cycle double covers, 
we decide to perform only safe reductions. 
Hence, the set $R$ introduced in the definition of mixed graphs is not needed. 
We observe that to get a directed cycle double cover of a cubic graph
$G$ it is sufficient to consecutively perform safe reductions, starting
 by safely reducing a subset, say $S$, of the vertex set of~$G$.
In other words, let $(V, E, A)$ be the mixed graph obtained 
from $G$ by a safe reduction of $S$. Then this safe reduction
of $S$ along with consecutive safe reductions 
of subsets $V_1,\ldots, V_k$ that partition $V$ 
construct a directed cycle double cover of $G$. 

\subsection*{Obstacles}  Which are the configurations 
that do not allow us to perform safe reductions? 
We refer to them as {\em obstacles}. We first observe
that if $S \subset S' \subset V$, where $V$ is the vertex
set of a mixed graph, and no safe reduction of $S$ exists, then no safe
reduction of $S'$ exists. 
For instance, a bridge is an obstacle: 
if a mixed graph with vertex set $V$ has an edge whose deletion 
separates $V$ into two sets with no edge or arc between them, 
then there is no safe reduction of $V$.
Another basic notion of our reasoning is a generalization of a bridge which we call {\em cut-obstacle}.
Let $S$ be a subset of vertices of a mixed graph, and let there $E_S$ and 
$A_S$ denote the sets of edges and arcs, respectively, with exactly one end-vertex in $S$. 
We say that $S$ is a {\em cut-obstacle} if there is no set $P$ of pairs in  $E_S \cup A_S$
such that each edge of $E_S$ is in exactly two pairs of $P$,
each arc of $A_S$ is in exactly one pair of $P$ and 
no pair of $P$ contains two arcs. 

We note that bridges are cut obstacles. In addition, an important example of
a cut obstacle is formed by a subset $S$ of vertices such that the number of edges with exactly one end in 
$S$ is strictly less than twice the number of directed edges with exactly one end in $S$.

Aside of cut obstacles, there is a wide variety of other 
concrete obstacles to the existence of safe reductions. 
However, it appears hopeless to analyze and keep track of all of them. 
This leads to a natural question: 

{\em 
Is there a class $\mathcal C$ of cubic bridgeless graphs such that:
(1) the dcdc conjecture may be reduced to the dcdc conjecture for $\mathcal C$, and
(2) for each graph $G$ of $\mathcal C$, the existence of consecutive safe reductions which 
do not create cut obstacles leads to the existence of a directed cycle double cover?}

In this work we propose the class of lean fork-graphs as a candidate for such class $\mathcal{C}$.
\subsection*{Fork graphs} 
The basic structures for the construction of {\em fork graphs} are contained in
the {\em fork-collection}. The fork-collection, denoted by $\mathcal{F}$,
consists of the \emph{$i$-big-forks} for every $i\geq 1$, the \emph{p-fork}, the \emph{fork}, the \emph{star fork}, 
 the \emph{subfork} 
and the \emph{dot}.
The p-fork, the fork and the star fork are depicted
in Figure~\ref{fig.subforks},
while a subfork is simply a pair of vertices 
connected by an edge and a dot is an isolated vertex. 
The 1-big-forks are depicted in Figure~\ref{fig.bigfork}. 
If $B$ is 1-big-fork, we let $C(B)=\{x,a,y,b'\}$ be the \emph{connecting set}
of $B$.
For $i\geq2$, each $i$-big-fork $B$ is obtained from a $(i{-}1)$-big-fork $B'$ 
and a star fork $T$ by connecting two leaves of $T$ to two vertices
of degree at most two of $B'$ with the following restriction: if we connect to a vertex of degree 2,
then it cannot belong to $C(B')$. We let $C(B)=C(B')\cup\{v\}$, where $v$
is the remaining leaf of $T$.
This operation is well explained in Figure~\ref{fig.bigfork},
since each 1-big-fork is obtained from a fork and a star fork in exactly the same way.
We refer to 1-big-forks simply as big-forks.
Furthermore, the {\em exclusive fork-collection} $\mathcal E$ is a subset of $\mathcal F$ 
that contains all members of  $\mathcal F$ but the fork.

\begin{figure}[h]
\centering
\subfigure[star fork]
{
\ifpdf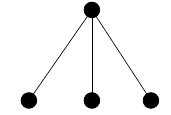\else\input{starfork.ps_t}\fi
  \label{fig.starfork}
}
\qquad
\subfigure[fork]
{
\ifpdf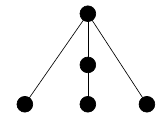\else\input{fork.ps_t}\fi
 \label{fig.fork}
}\qquad
\subfigure[p-fork]
{
\ifpdf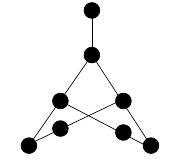\else\input{petf.ps_t}\fi
 \label{fig.petfork}
}
\caption{fork-type graphs}
\label{fig.subforks}
\end{figure}

\begin{figure}[h]
\centering
\subfigure[]
{
\ifpdf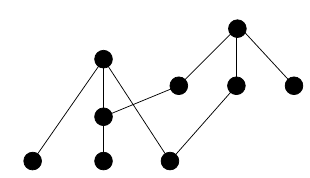\else\input{big_fork_1.ps_t}\fi
 \label{fig.bf1}
}\qquad
\subfigure[]
{
\ifpdf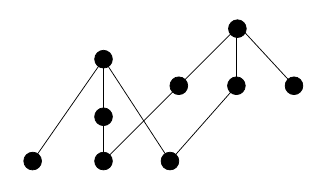\else\input{big_fork_2.ps_t}\fi
  \label{fig.bf2}
}\qquad
\subfigure[]
{
\ifpdf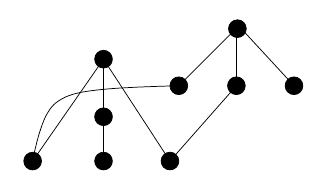\else\input{big_fork_3.ps_t}\fi
  \label{fig.bf3}
}
\caption{Three kinds of big-forks. Note that a big-fork consists of the union of a fork and a star fork by means of the 
addition of two new edges.}
\label{fig.bigfork}
\end{figure}

Given a member $L$ of the fork-collection, a {\em bold $L$} is obtained 
from $L$ by adding several extra half-edges and edges 
following the next four rules. (i) If $L$ is the fork, the star fork, or the subfork
we obtain a bold $L$ by adding one half-edge to each leaf of~$L$. 
(ii) If $L$ is the p-fork, a bold $L$ is obtained by 
adding one half-edge to the leaf of $L$ and to each vertex $x$, $y$ (see Figure~\ref{fig.petfork}).
(ii) Add two or three half-edges to the dot to obtain a bold dot.
(iii) For each $j\geq 1$, if $L$ is a $j$-big-fork, we obtain a 
 bold $L$ by adding a half-edge to each vertex from the connecting
set of $L$ and add a set, possibly empty, of disjoint edges and half-edges so that 
the degrees of the vertices of the bold $L$ are at most 3.

\subsubsection*{Definition of fork graphs}
We say that a cubic graph $G$ is a {\em fork graph} if there is a sequence $G_0, 
\ldots, G_n$ of 2-connected graphs 
so that $G_0$ is a triangle, $G_n= G$ and $G_{i}$ is obtained 
from $G_{i-1}$ by connecting to its vertices of degree $2$ the half-edges
of a bold $L_i$, where $L_i$ is from the exclusive fork-collection; therefore,
half edges of a bold $L_i$ become edges of $G_i$. 
In addition, we allow that for at most one $j$, $L_j$ is the fork; the fork is depicted 
in Figure~\ref{fig.fork}. 
Moreover, we can perform several Y${-}\Delta$, $\Delta{-}$Y transformations;
that is,~replacement of a vertex by a triangle, and vice-versa.
We say, in the situations described above, that $G_{i}$ is obtained from $G_{i-1}$ by {\em addition} of a bold $L_i$.

\begin{example}[\underline{Petersen's graph}]
Let $G_1$ be the graph obtained from a triangle by addition of a bold p-fork such that $G_1$ has exactly
3 vertices of degree two. Let $G_2$ be the cubic graph obtained from $G_1$ by addition of a bold dot.
The graph obtained from $G_2$ by performing one $\Delta{-}$\emph{Y} transformation at the initial triangle 
is the Petersen's graph. 
\end{example}

\smallskip

Our first main contribution concerns cut-type sufficient conditions for the existence of a safe reduction 
and therefore the existence of a directed cycle double cover conjecture.

\begin{theorem}
\label{thm.main1}
Let $G$ be a fork graph and $G_0, \ldots, G_n$ be its building sequence.
Let $i\leq n$ and $G_i$ be obtained from $G_{i-1}$ by addition of a bold $L_i$.
If $V(G){-}V(G_i)$ can be safely reduce in $G$ 
and $G'_i$ denotes the obtained mixed graph, then the following two statements hold.
\begin{enumerate}
 \item  If $L_i$ is not a $j$-big-fork for all $j{\geq}1$, then $V(L_i)$ 
can be safely reduced in $G'_i$. 
 \item If $L_i$ is a $j$-big-fork for some $j{\geq}1$ and $L_i$ is not a cut-obstacle in $G'_i$,
 then $V(L_i)$ can be safely reduced in~$G'_i$. 
\end{enumerate}
\end{theorem}

\subsection*{Lean fork-graphs} 
Let $G$ be a fork graph and $G_0, \ldots, G_n$ be its building sequence.
For $i\leq n$ and $j\geq 1$ , let $G_i$ be obtained from $G_{i-1}$ by addition of a bold $j$-big-fork $L_i$.
Since bold $L_i$ has at most $j+4$ vertices of degree 2, we observe that there are
at most $j+4$ vertex-disjoint paths from $V(L_i)$ to $V(G_{i-1})$ using
only edges from $E(G){-}E(G_{i})$. The connectivity property of a lean fork graph
is that instead of $j+4$, at most $j+3$ vertex-disjoint paths are allowed. Namely, $G$ is said to 
be {\em lean} if for each $i$ such that $G_{i}$ 
is obtained from $G_{i-1}$ by addition of a bold $j$-big-fork $L_i$ for some $j{\geq} 1$, there are at most $j+3$ 
vertex-disjoint paths from $V(L_i)$ to $V(G_{i-1})$ using 
only edges from $E(G){-}E(G_{i})$.

 Our next main contribution is confirmation of Jaeger's conjecture for the class of all lean fork-graphs.

\begin{theorem}
\label{thm.main2}
{The{ }directed cycle double cover conjecture holds for all lean~fork{-}graphs.}
\end{theorem} 

Finally, we show that each cubic bridgeless graph $G$ is naturally embedded in a lean~fork-graph.

\begin{theorem}\label{thm:contsubofcub}
For every cubic bridgeless graph $G$ there exists a lean fork graph that contains a subdivision of $G$
as an induced subgraph.
\end{theorem}


In the following section, we explain basic technical tool of our reasoning, namely the hexagon graphs, 
introduced in \cite{jkl}.

\vspace{1ex}
\section{Hexagon graphs}
We refer to the complete bipartite graph $K_{3,3}$ as a \emph{hexagon} and
say that a bipartite graph $H$ has a hexagon $h$ if $h$ is a subgraph of $H$.
For a graph $G$ and a vertex $v$ of $G$, let $N_G(v)$ denote the set of neighbors of $v$ in $G$.

Let $G$ be a cubic graph with vertex set $V$ and edge set $E$.
A hexagon graph of $G$ is a graph $H$ obtained from $G$ following the rules:
\begin{enumerate}
\item We replace each vertex $v$ in $V$ by a hexagon $h_v$ so that  
      for every pair $u$, $v \in V$, if $u \neq v$, then $h_u$ and $h_v$ are vertex disjoint.     Let  $\{V(h_v): v \in V \}$ be the vertex set of $H$.
\item For each vertex $v\in V$, let $\{v_i: i \in \mathbb{Z}_6\}$ denote the vertex set of $h_v$ 
      and $\{v_{i}v_{i+1}, v_{i}v_{i+3}: i \in \mathbb{Z}_6\}$ its edge set. 
      With each neighbor $u$ of $v$ in $G$, we associate an index $i_{v(u)}$ 
      from the set $\{0,1,2\} \subset \mathbb{Z}_6$ so that if 
      $N_G(v)=\{u,w,z\}$, then $i_{v(u)}$, $i_{v(w)}$, $i_{v(z)}$ are pairwise distinct. 
\item \label{def.hex_third} (See Figure~\ref{fig:hexagons}). Let $X= \cup_{v\in V} \{v_{2i} :i\in \mathbb{Z}_6\}$ and 
      $Y= \cup_{v\in V} \{v_{2i+1} :i\in \mathbb{Z}_6\}$.
      We replace each edge $uv$ in $E$ by two vertex disjoint edges $e_{uv}$, $\bar{e}_{uv}$ so that 
      if both $v_{i_{v(u)}}$, $u_{i_{u(v)}}$ belong to either $X$ or $Y$, then
      $e_{uv}= v_{i_{v(u)}} u_{i_{u(v)}+3}$, $\bar{e}_{uv}= v_{i_{v(u)}+3} u_{i_{u(v)}}$.
      Otherwise, 
      $e_{uv}= v_{i_{v(u)}} u_{i_{u(v)}}$, $\bar{e}_{uv}= v_{i_{v(u)}+3} u_{i_{u(v)}+3}$.
      Let $E(H)=\{E(h_v): v \in V\} \cup \{ e_{uv}, \bar{e}_{uv} : uv \in E\}$.
      \end{enumerate} 

\begin{figure}[h] 
 \centering 
 \ifpdf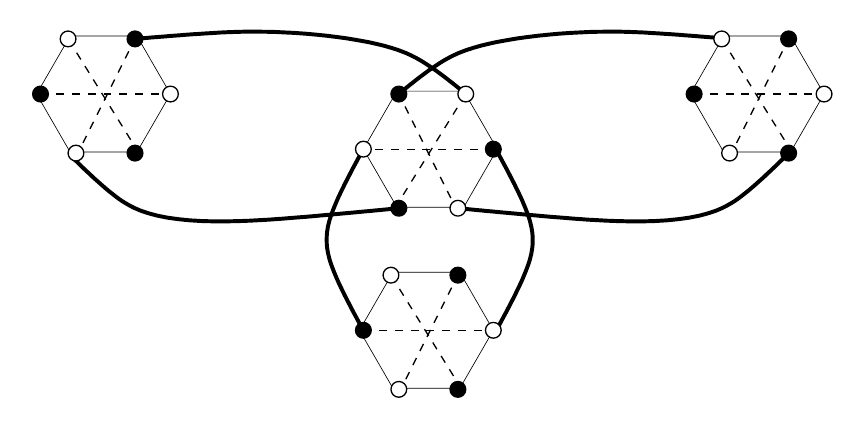\else\input{hexagons.ps_t}\fi
 \caption{\small{Local representation of a hexagon graph $H$
 of a cubic graph $G$. The hexagon $h_v$ is associated with vertex $v$, where $N_{G}(v)=\{u,w,z\}$.
 Red edges are depicted as red lines, blue edges are depicted as blue lines and 
 white edges as black lines. The set $X$ is represented by white vertices
 and the set $Y$ by black vertices.}}
\label{fig:hexagons}
\end{figure}

We say that $h_v$ is the hexagon of $H$ associated with the vertex $v$ 
  of $G$ and that $\{h_v: v \in V \}$ is the \emph{set of hexagons of $H$}. 
We shall refer to the set of edges $\bigcup_{v\in V} \{v_{i}v_{i+3}: i\in\mathbb{Z}_6\}$ as the 
set of \emph{red edges of $H$}, to the set of edges $\{ e_{uv}, \bar{e}_{uv} : uv \in E\}$
as the set of \emph{white edges of $H$}, and finally to the set of edges
$\bigcup_{v\in V} \{v_{i}v_{i+1}: i\in\mathbb{Z}_6\}$ as the set of \emph{blue edges of $H$}
 (see Figure~\ref{fig:hexagons}). Moreover, we shall 
 say that a perfect matching of $H$ containing only blue edges is a \emph{blue perfect matching}.  
In the rest of this work, if $x$ is a vertex, say $v_i$, of a hexagon, then $\bar{x}$ denotes $v_{i+3}$.

 Let $G$ be a cubic graph and $H$ be a hexagon graph of $G$. We observe two important properties: 
(i) $H$ is bipartite, and (ii) if $H'$ is another hexagon graph of $G$, then $H$ and $H'$ are isomorphic.

In the next paragraphs we briefly recall a combinatorial representation of 
  embedding of graphs on closed orientable surfaces, namely \emph{rotation systems}, 
  and describe the embeddings encoded by the blue perfect matchings. 
  
Let $G$ be a graph. For each $v \in V(G)$, let $\pi_{v}$ be a cyclic permutation of the
edges incident with~$v$. 
A collection $\pi {=} \{\pi_{v} : v \in V(G)\}$ is called a 
\emph{rotation system} of $G$. Edmonds~\cite[\S 3.2]{opac-b1131924} proved 
that each such a $\pi$ encodes an embedding of $G$ 
on a closed orientable surfaces with face boundaries 
$e_1e_2 \cdots e_k$ such that $e_i{=}v^iv^{i+1} \in E(G)$, 
$\pi_{v^{i+1}}(e_i)= e_{i+1}$ 
$e_{k+1}=e_1$ and $k$ minimal. 
 
Let $M$ be a blue perfect matching of $H$ and let $W$ be the set of white edges of $H$. 
Each cycle $C$ in $M\Delta W$ induces a subgraph in $G$ defined by the set of edges
$ \{uv \in E(G): e_{uv} \in C \, \text{or} \,\, \bar{e}_{uv} \in C\}$.
The following lemma follows via a natural bijection between blue perfect matchings and rotation systems.

\begin{lemma}\label{lemma:hexagon-embedding} 
Let $G$ be a cubic graph, $H$ the hexagon graph of $G$, and $W$ the set of white edges of $H$.
Each blue perfect matching $M$ of $H$ encodes 
an embedding of $G$ on a closed orientable surface with set of face boundaries
the set of subgraphs of $G$ induced by the cycles in $M\Delta W$. 
Moreover, the converse also holds. 
\end{lemma}

In~\cite{jkl}, we establish the following approach to the directed cycle double cover conjecture.

\begin{proposition}\label{prop:dualloops}
Let $G$ be a cubic graph, $H$ the hexagon graph of $G$, $M$ a blue 
perfect matching of $H$, and $W$ the set of white edges of $H$.
 The embedding of $G$ encoded by $M$ has a dual loop if and only if there is a cycle
 in $M\Delta W$ that contains  both end-vertices of a red edge.
\end{proposition}

In the same work~\cite{jkl}, we prove the following structural result
regarding hexagon graphs. We recall that braces, along with bricks, form the basic building blocks
 of the perfect matching decomposition theory~\cite{MR904405}.

\begin{theorem}
\label{thm.brace}
Let $G$ be a cubic graph. Then the hexagon graph $H$ of $G$ is a brace if and only if $G$ is bridgeless.
\end{theorem}

The following section (Section~\ref{sec.proofthe2and3}) is divided into two parts: 
the first one shows how Theorem~\ref{thm.main1} implies Theorem~\ref{thm.main2}
and in the second part we prove the statement of Theorem~\ref{thm:contsubofcub}.
Finally, in Section~\ref{sec:prooftheo1} we present the proof of Theorem~\ref{thm.main1}
in the context of hexagon graphs.

\section{DCDC and richness of lean fork graphs}\label{sec.proofthe2and3}
In this section we show how Theorem~\ref{thm.main1} implies Theorem~\ref{thm.main2} 
and  discuss the proof of Theorem~\ref{thm:contsubofcub}.

\subsection*{Proof of Theorem~\ref{thm.main2}} 
\label{sub.2}
Let $G$ be a lean fork-graph. 
Hence, there is a sequence  $G_0, \ldots, G_n$ of 
2-connected graphs such that $G_0$ is a triangle, $G_n= G$, and  
for $i\le n$, $G_{i}$ is constructed from $G_{i-1}$ by adding a bold $L_i$, where 
$L_i$ is a member of the exclusive fork-collection and it is the fork
at most once.

Using Theorem \ref{thm.main1}, we only need to show that in subsequent safe 
reductions of vertex sets of bold $L_i$'s, we do not create a cut-obstacle formed by the vertex set of an added bold $j$-big-fork,
for some~$j\geq 1$. 

We assume, for the sake of contradiction, that 
for some $ i\le n$ and $j\geq 1$, $G_{i}$ is obtained from $G_{i-1}$ by addition of a 
bold $j$-big-fork $B$ and $V(B)$ is a cut-obstacle in $G'_{i}$,
where $G'_{i}$ denotes a mixed graph obtained by safely reducing $V(G){-}V(G_i)$.
This implies 
that the number of directed edges in $G'_i$ with exactly one end-vertex in $V(B)$ is more 
than $2(j{+}3)$. Hence, this number is exactly $2(j{+}4)$, because a bold $j$-big-fork has at most $j{+}4$ vertices of 
degree 2 in $G_{i}$ (see Figure~\ref{fig.bigfork} for $j=1$).


These $2(j{+}4)$ directed edges are obtained by the reduction of $V(G){-}V(G_{i})$. 
Let $G'$ be the graph obtained from $G$ by deleting all edges in $E(G_{i})$.
By definition of safe reductions, the digraph $D$
obtained from $G'$ by replacing each edge by two oppositely directed edges, has $2(j{+}4)$ directed
edge-disjoint paths between $V(B)$ and $V(G_{i-1})$. 
Hence, $D$ has no set of strictly less than $2(j{+}4)$ directed edges
which completely separates $V(B)$ from $V(G_{i-1})$. Consequently, $G'$ has no set of at most
 $j{+}3$ edges which completely separates $V(B)$ from $V(G_{i-1})$. 
But this means, by the Menger's theorem, that $G'$ has at least $j{+}4$ edge-disjoint paths between $V(B)$
and $V(G_{i-1})$. Since each vertex of $G'$ has degree at most three, these paths are also vertex disjoint. This contradicts 
the assumption that $G$ is lean. \hfill $\square$

\subsection*{Proof of Theorem~\ref{thm:contsubofcub}} 
\label{sub.3}

Given a cubic bridgeless graph $G$, the construction of the lean fork-graph $\tilde{G}$
that contains a subdivision of $G$ as an induced subgraph can be split into two steps:
\begin{description}
 \item[Step 1:] We create a lean fork-graph $\tilde{G}^1$ with arbitrarily many vertices of degree 2:
 we perform this task by constructing the lean fork-graph with defining sequence 
 $\tilde{G}^1_0, \tilde{G}^1_1, \ldots, \tilde{G}^1_m$,
 where $\tilde{G}^1_0$ is the triangle, $\tilde{G}^1_1$ is obtained from $\tilde{G}^1_0$ by adding 
 a bold fork and for each $i > 1$, $\tilde{G}^1_i$ is obtained from $\tilde{G}^1_{i-1}$
 by adding a bold $(i{-}1)$-big-fork. We note that $\tilde{G}^1$ is a lean fork-graph
 with exactly $m+3$ vertices of degree two.
 
  \item[Step 2:] In the second step we obtain $\tilde{G}$ from $\tilde{G}^1$ by sequentially adding bold subforks and 
   bold dots following an ear-decomposition of $G$.   
   Let~$(G_0, G_1, \ldots, G_l; P_1, \ldots, P_l)$ be an ear decomposition of $G$,
   where $G=G_l$, $G_0$ is a cycle of $G$ and  
   $G_{i}$ is obtained from $G_{i-1}$ connecting two vertices of $V(G_{i-1})$ 
   by a path $P_i$ such that $E(P_i) \cap E(G_{i-1}) = \emptyset$ 
   and $|V(P_i) \cap V(G_{i-1})|=2$, for each $i \in \{1, \ldots,l\}$. In order to obtain
   $\tilde{G}$ from $\tilde{G}^1$ with the property that $\tilde{G}$  has
   a subdivision of $G$ as an induce subgraph, we first obtain $\tilde{G}_0$
   from $\tilde{G}^1$ with the property that $\tilde{G}_0$ contains a subdivision of $G_0$ and 
   then for each  $i \in \{1, \ldots,l\}$, we obtain  $\tilde{G}_i$ from  $\tilde{G}_{i-1}$
   with the property that $\tilde{G}_i$ contains a subdivision of $G_i$. 
   This procedure is best explained by means of an example, see Figure~\ref{fig.spider}.   
   Since $\tilde{G}^1$ is lean and by the construction of $\tilde{G}$ from $\tilde{G}^1$, 
   we have that $\tilde{G}$ is a lean fork-graph.  \hfill $\square$

\begin{figure}[h]
\centering
\ifpdf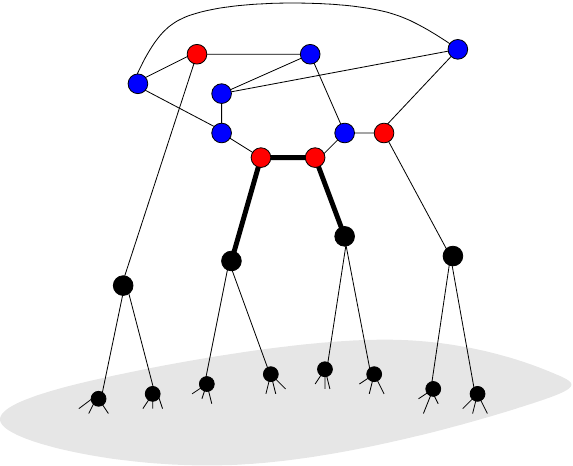\else\input{spider.ps_t}\fi
\caption{Construction of $\tilde{G}$ from $\tilde{G}^1$ in the case that $G$ 
is the complete bipartite graph on 6 vertices. The vertex set of $G$ is colored by blue and the
vertices produced by the subdivisions are colored by red.
Moreover, note that $L_1, L_2, L_3, L_4$ are subforks and $L_5, L_6$ are dots.}
\label{fig.spider}
\end{figure}
\end{description}

\section{Proof of Theorem~\ref{thm.main1}}

\label{sec:prooftheo1}

In this work, we perform safe reductions of subsets of vertices of cubic graphs.
Recall that this leads us to the notion of mixed graphs (see Section~\ref{sec:ideasandresults}).
We suggest to study these reductions in the context of hexagon graphs.
Then it is necessary to describe mixed graphs in terms of hexagon graphs.
We refer to these new structures as \emph{pseudohexes}, and introduce
them and the corresponding concept of \emph{safe reductions of pseudohexes} in the following paragraphs.

The statement of Theorem \ref{thm.main1} follows from Lemma~\ref{lemma:corrsafred}
and Theorems~\ref{thm.fearr}, \ref{theo.sub-fork-ear}, \ref{thm.j-bigforks}, \ref{theo.p-fork}.

\subsection{Pseudohexes} 
Let $K$ be a bipartite graph with vertex set $V(K)$ and edge set $E(K)$ such that 
edges are colored blue, red and white, and there may be parallel edges but no loops. 
We let $B(K)$ denote the set of its blue edges, $R(K)$ the set of its red edges, and  
$W(K)$ the set of its white edges. We call $K$ {\em pseudohex} if it is empty or the 
following three properties are satisfied:
\begin{itemize}
\item
$B(K)$ induces disjoint hexagons covering all the vertices of $K$; we call the set of hexagons induced
by $B(K)$ the hexagons of $K$. 
\item
$K$ contains all the edges $\{x,\bar{x}\}$ for each vertex $x$ of $K$; all these edges are red. 
Moreover, $K$ may have additional red edges. 
\item
$W(K)=E(K) {-} (B(K)\cup R(K))$ is a perfect matching of $K$.
\end{itemize}

We still need to introduce some extra definitions.
A white edge $e = \{x,y\} \in W(K)$ is said to be {\em real} if $K$ contains 
also the white edge  $\bar{e}= \{\bar{x},\bar{y}\}$
and the only red edges adjacent to $e$ or $\bar{e}$ 
are $\{x,\bar{x}\}$ and  $\{y,\bar{y}\}$.
Clearly $e$ is real if and only if $\bar{e}$ is real. 
We denote by $W_r(K)$ the set
of the real white edges of $K$. 
In addition, we say that the white edges in  $W_d(K)=W(K) {-} W_r(K)$ are 
{\em derived} and that two white edges are {\em red-connected} if there is a red edge adjacent to both of them.

We remark that pseudohexes correspond to mixed graphs; the set of hexagons of
a pseudohex corresponds to the vertex set of the mixed graph, the sets of the real white
edges and the derived white edges correspond to the sets of the edges and the arcs, respectively,
of the mixed graph. 
Moreover, red-connected pairs of derived white edges correspond 
to pairs of arcs in the set $R$ introduced in the definition of mixed graphs.

\begin{definition}[Reduction of a hexagon] \label{def:reductofh}
Let $K$ be a pseudohex, $h$ a hexagon of $K$ and $N$ a perfect matching of $h$. We define the 
{\em reduction} of $h$ by $N$ as follows.
For each of the three paths $P_i, i=1,2,3$, consisting of one edge of $N$ and the white edges adjacent 
to this edge we introduce a new white edge $e_i$ between the end-vertices of $P_i$
whenever $P_i$ is not a cycle of length 2. If $P_i$ is not a cycle of length 2, then 
for each red edge $\{u,w\}$ attached to an interior vertex 
$w$ of $P_i$ we introduce new red edge $\{u,w'\}$ where $w'$ is the vertex of $e_i$ of the same bipartition 
class of $K$ as $w$. If $P_i$ is a cycle of length 2, then for each red edge $\{u,w\}$ attached to an interior vertex 
$w$ of $P_i$ we introduce a new red edge $\{u,\bar{u}\}$.
Finally, we delete the vertices of $h$.
The paths $P_i,  i= 1,2,3$, are said to be \emph{contracted}. 
\end{definition}

We observe that a reduction never creates new real white edges.
We say that a set of hexagons $\{h_1, \ldots, h_l\}$ can be {\em safely reduced} (or the reduction is {\em safe}) 
if there are perfect matchings $N_1, \ldots, N_l$ of $h_1, \ldots, h_l$, respectively, such that
each contracted path of the joint reduction of $h_1, \ldots, h_l$ by $N_1, \ldots, N_l$, respectively,  
has at most one white edge which is not real, in other words 
at most one derived white edge. 
The next statement follows from the facts that the set of derived white edges in pseudohexes is equal
to the set of the arcs in mixed graphs and that the reduction of a hexagon by a perfect matching in a pseudohex
corresponds to wiring a vertex in a mixed graph.

\begin{lemma}\label{lemma:corrsafred}
A safe reduction in a pseudohex corresponds to safe reduction in a mixed graph.
\end{lemma}



Let $K$ be a pseudohex.
It is natural to associate a graph $G^K$ with $K$. 
Recall that $K$ corresponds to a mixed graph.
The vertex set of $G^K$ is the set
of the hexagons of $K$. For $u, v$ vertices of $G^K$, the set $\{u, v\}$ is an edge
of $G^K$ if and only if there is a pair of real white edges connecting 
the hexagons $h_u, h_v$ in $K$. 
We know by definition of a pseudohex that each vertex of $G^K$  has degree at most three.
In other words, $G^K$ is the graph induced by the mixed graph that corresponds
to the pseudohex $K$.
A subgraph of $K$ is called {\em end} if it consists of a red edge parallel to a white edge. 
If $K$ contains an end as a subgraph we say that $K$ has an end.
We say that $K$ is {\em proper} if $K$ has no end and $G^K$ is a 2-connected graph without cycles of length~2.

We observe that any mixed graph obtained from a fork graph by a sequence of safe reductions
of set of vertices of bold members of the fork-collection can be represented by a proper pseudohex.
Therefore, in order to prove Theorem~\ref{thm.main1}, 
we are interested in obtaining safe reductions of set of hexagons of proper pseudohexes,
where the set of hexagons correspond to the vertex-set of the fork-type graphs.
In Subsections~\ref{subsec:technicalstuff1},~\ref{sec.subforks_subears} and~\ref{sub:bigfork}  
we study the aforementioned reductions.

\subsection{Reduction of forks and 3-ears on pseudohexes} \label{subsec:technicalstuff1}

The aim of this section is to study safe reductions of forks on pseudohexes.
However, we also establish results that are used to handle many of the proofs included in later sections.
For this sake, we introduce a new fork-type graph that we call the \emph{3-ear}.
The 3-ear consist of a path on three vertices. 
We refer to the union of the fork-collection and the 3-ear
as the \emph{extended} fork-collection.
The definition of a bold 3-ear is analogous to the one of bold fork, star fork and subfork.

Let $K, K'$ be proper pseudohexes such that $G^K$ is obtained from $G^{K'}$ by addition of 
a bold $L$, where $L$ is a member of the extended fork-collection.
We denote by $L_K$ the subset of the hexagons of $K$ corresponding to the vertices of $L$;
consequently, $V(L_K)$ denotes the set of the vertices of $K$ corresponding to the hexagons in $L_K$.
We say that $K$ is obtained from $K'$ by \emph{$L$-addition}. 

We refer to the subset of real white edges (derived white edges, respectively) of $K$  
with at least one end vertex in $V(L_K)$
as the {\em $L$-edges} ({\em $L$-no-edges}, respectively) of $K$.
A pair of $L$-no-edges is called {\em potential} 
if each edge of the pair is incident to exactly one vertex of $V(L_K)$,
and these two vertices belong to different bipartition classes of $K$.

In the case that $L$ is the fork or the 3-ear, 
we say that $K$ has a {\em $L$-obstacle} if each  
$L$-no-edge of $K$ is incident to exactly one vertex of~$V(L_K)$.
We remark that the notion of $L$-obstacle is analogous to the notion of cut-obstacle for mixed graphs.
Observation~\ref{o.obss} illustrates the concept of $L$-obstacles. 

\begin{observation}
\label{o.obss}
If a proper pseudohex $K$ has a $L$-obstacle and all potential pairs of $L$-no-edges are red-connected,
then each reduction of the hexagons corresponding to the vertices of $L$ creates an end.
\end{observation}
\begin{proof}
We prove the observation for a $F$-obstacle, where $F$ is the fork. 
The proof of the statement in the case that $L$ is a 3-ear. 
Hence, we assume that $K$ has a $F$-obstacle and let $\tilde{K}$ denote the pseudohex 
obtained from $K$ by some reduction of 
$F_K$. 

Each new derived white edge of $\tilde{K}$ is obtained by contracting a path that contains two  
  white edges from the set $S$ of the white edges of $K$ with exactly one end-vertex in $V(F_K)$. 
Moreover, the end-vertices (contained in $V(F_K)$) of these two white edges 
belong to different bipartition classes of $K$.

The set $S$ is formed by all eight $F$-no-edges of $K$ and a subset consisting of six $F$-edges.
Hence, necessarily a potential pair of $F$-no-edges belong to the same contracted path of the reduction; 
  but each potential pair of $F$-no-edges of $K$ is red-connected, and thus $\tilde{K}$ has an end.
\end{proof}

We recall that the reduction of a hexagon is {\em safe} if each contracted path has at most one derived 
white edge (see Definition~\ref{def:reductofh}). We say that 
a pseudohex has a \emph{correct reduction} if it is possible to 
reduce all its hexagons without creating an end.  

The following converse of Observation~\ref{o.obss} is not difficult to prove.

\begin{observation}
\label{o.ob2}
Let $K$ be a proper pseudohex with a $L$-obstacle, where $L$ is either the fork, or the 3-ear.
If there exists a potential pair $e_1, e_2$ of $L$-no-edges of $K$ that are not red-connected
then, there is a reduction of $L_K$ that creates no end,
and such that all but one of the contracted paths contain at most one $L$-no-edge. 
Moreover, the contracted path that contains more
than one $L$-no-edge contains exactly two $L$-no-edges; namely, $e_1, e_2$.
\end{observation}
\begin{proof}
We prove the observation first for a $P$-obstacle, where $P$ is the 3-ear. 
We assume that there exists a potential pair $e_1, e_2$ of $P$-no-edges of $K$ that are not red-connected.
Without loss of generality two cases arise: the case that $e_1$ and $e_2$ have end vertices in 
  $V(h_x \cup h_y)$ and the case that $e_1$ has an end vertex in $V(h_x)$
  and $e_2$ has an end vertex in $V(h_z)$. Both cases are worked out in Figure~\ref{fig.4}.
\begin{figure}[h]
\centering
\subfigure[$e_1, e_2$ have end-vertices in $V(h_x \cup h_y)$.]
{
\ifpdf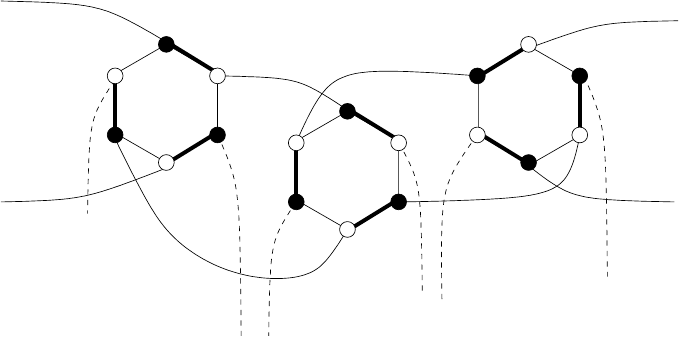\else\input{4-1-new.ps_t}\fi
 \label{fig.41}
}\qquad
\subfigure[$e_1$ has end-vertex in $V(h_x)$
and $e_2$ in $V(h_z)$.]
{
\ifpdf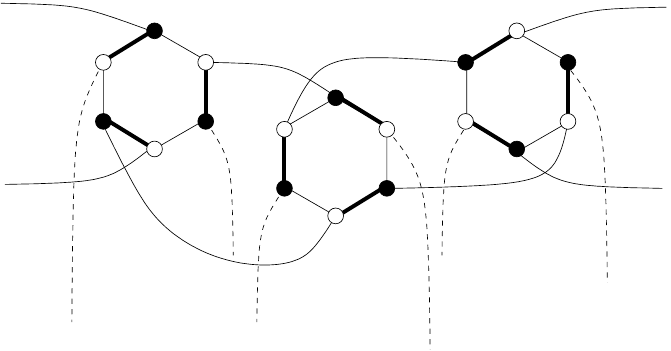\else\input{4-2-new.ps_t}\fi
  \label{fig.42}}
\caption{P-no-edges are represented by dotted lines. Perfect matchings leading to a reduction of
$\{h_x, h_y, h_z\}$ without 
ends are depicted by thicker lines.
In the case that $e_1 =e_2$, the reduction by the same perfect matchings is safe.
}
\label{fig.4}
\end{figure}

Secondly, we assume that $K$ has a $F$-obstacle, where $F$ is the fork. 
In this proof, we use names of vertices, edges and hexagons of $F_K$ from Figure~\ref{fig.abad}.
Also, using Figure~\ref{fig.abad}, we denote by $M_a$ and $M_b$ the perfect matchings 
indicated by thicker lines on hexagons $h_a$ and $h_b$, respectively. Moreover,
let $N_a$ and $N_b$ denote the perfect matchings of $h_a$ and $h_b$ 
that are complements of $M_a$ and $M_b$, respectively.

In order to use the first part of this proof regarding the 3-ear, 
we observe that the reduction of $h_a, h_b$ by $M_a, N_b$ (by $N_a, M_b$, respectively) is safe
and generates a $P$-obstacle in the resulting pseudohex, where $P$ is the 3-ear induced by the set of vertices $\{x, z, y\}$
corresponding to the set of hexagons $\{h_x, h_z, h_y\}$.

  Let $e_1$, $e_2$ be a potential pair of $F$-no-edges of $K$ that is not red-connected.
  Without loss of generality, we distinguish three cases depending on the hexagons to which the end-vertices
  of the pair of edges $e_1$, $e_2$ belong to. 
\begin{itemize}
\item[(i)] Both edges $e_1, e_2$ have an end-vertex in $V(h_x \cup  h_y)$.
Then we can safely reduce $h_a, h_b$ by $M_a, N_b$. 
We get a $P$-obstacle, where the non red-connected pair $e_1, e_2$ of $F$-no-edges becomes a pair of $P$-no-edges that is 
not red-connected and we can use the first part of this proof to complete the argument.

\item[(ii)] Edge $e_1$ is incident to a vertex in $V(h_x \cup h_y)$
and $e_2$ is incident to a vertex in $V(h_a \cup h_b)$. 
Then $e_2$ is incident to a vertex from the set $\{r_1, r'_1, r_2, r'_2\}$; see Figure~\ref{fig.abad}.
If $e_2$ is incident to $r_1$ then reduction of $h_a, h_b$ by $M_a, N_b$ generates a 
$P$-obstacle where $e_1$ and the derived white edge
starting at $w'$ (that is, the edge obtained by reduction of a path containing $e_2$)  are not red-connected. 
Again, we use the first part of this proof to complete the argument. 
The remaining cases that $e_2$ is incident to 
$r'_1$, $r_2$ or $r'_2$ can be worked out in the same way.

\item[(iii)] Both edges $e_1, e_2$ have an end-vertex in $V(h_a \cup h_b)$.
In the case that the pair $e_1, e_2$ is incident to $r_1, r'_1$ ($r_2, r'_2$, respectively), 
we reduce $h_a,h_b$ by $M_a, M_b$ ($N_a, N_b$, respectively).
In both cases a path consisting of real 
white edges and $e_1, e_2$ is contracted to a single derived white edge disjoint from $V(P_K)$. 
Moreover, the derived white edge incident to $w$ ($w'$, respectively) 
is obtained by contracting a path consisting only of real white edges;
therefore, this edge is not red-connected to the derived white edges incident to $V(h_x \cup h_y)$. 
Hence, both reductions generate a $P$-obstacle with a
  potential pair of $P$-no-edges that are not red-connected. Again, by the first part of this proof 
  the result holds.\end{itemize}
                   
\end{proof}

The following statement basically claims that it is always possible to safely reduce the vertex
set of the starting triangle of a fork-graph.  

\begin{observation}
\label{o.triangle}
If $K$ is a proper pseudohex with $G^K$ a cycle of length three, then the set of all hexagons
of $K$ has a correct (and safe) reduction. 
\begin{figure}[h]
\begin{center}
\ifpdf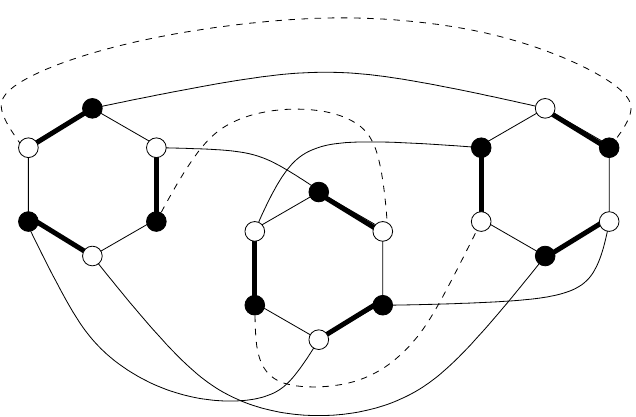\else\input{triangle.ps_t}\fi
\end{center}
\caption{Correct and safe reduction of a proper triangle: 
 reduction of each hexagon by the perfect matching represented by thicker lines.
 Derived white edges are depicted as dotted lines.}
\label{fig.triangle}
\end{figure}\end{observation}
\begin{proof}
Without loss of generality, we assume that $K$ has a derived white edge $e$ as in Figure \ref{fig.triangle};
there is no loss of generality, since $G^K$ is symmetric.
Then, since $K$ is proper, the remaining two derived white edges have to be 
configured as in Figure~\ref{fig.triangle}.
The perfect matchings indicated by thicker lines in Figure~\ref{fig.triangle} lead to a correct and safe reduction 
of the set of all hexagons
of $K$, since each contracted cycle (the edges of a contracted cycle alternate
white edges and edges from the indicated perfect matchings) of this reduction
has at most one derived white edge, and thus induces no red edge.
\end{proof}

In the next definition we present an extra obstacle for a safe reduction of a 3-ear.

\begin{definition}
\label{def.3ebad}
Let $K, K'$ be proper pseudohexes and $K$ obtained from $K'$ by 
  $P$-addition, where $P$ is the 3-ear. Let $\{x,z,y\}$ be the vertex
  set of $P$ with $x,y$ its leaves.
We say that $K$ has a {\em $P$-danger} if one $P$-no-edge has its end vertices in  $V(h_x\cup h_y)$, one 
  $P$-no-edge has one end vertex in  $V(h_x \cup h_y)$ and the other end vertex in $V(h_z)$, and 
  the remaining two $P$-no-edges have exactly one end vertex in $V(P_K)$. 

Moreover, we say that $K$ has a {\em $P$-bad} if it has a $P$-danger and the $P$-no-edge with both end vertices in
$V(h_x \cup h_y)$ is red-connected to the $P$-no-edge
that has exactly one end vertex in $V(P_K)$ and it belongs to $V(h_z)$.
An example of a $P$-bad is in Figure \ref{fig.bad}, where the two edges which must be red-connected are $e$ and $\g$.
\begin{figure}[h]
\begin{center}
\ifpdf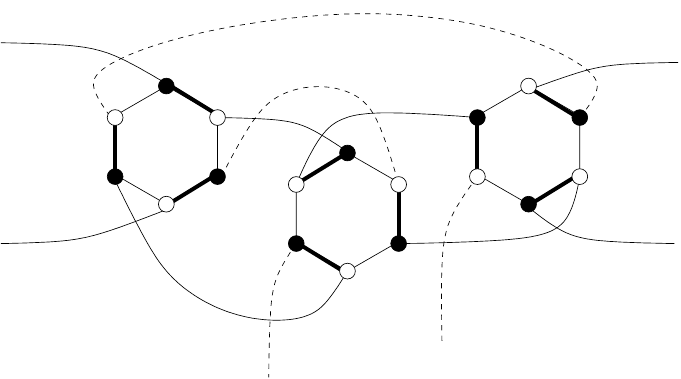\else\input{bad.ps_t}\fi
\end{center}
\caption{P-danger pseudohex; P-no-edges are depicted by dotted lines. If it is P-bad then edges $e,\g$ are red-connected. 
If $ e,\g$ are not red-connected, then perfect matchings depicted by thicker lines
lead to a reduction of $\{h_x, h_y, h_z\}$ that creates no end.}
\label{fig.bad}
\end{figure}
\end{definition}

\begin{theorem}
\label{thm.frsst}
Let $K, K'$ be proper pseudohexes and $K$ obtained from $K'$ by 
  $P$-addition, where $P$ is a 3-ear. 
If $K$ has neither a $P$-obstacle nor a $P$-danger, then $P_K$ can be safely reduced.

Moreover, if $K$ has a $P$-danger but not a $P$-bad then there is a reduction of $P_K$ creating no end and such that
each but one contracted path has at most one derived white edge, and the contracted path that has more than
one derived white edge contains exactly two derived white edges, namely $\g$ and  $e$ 
 (see Figure~\ref{fig.bad} and Definition~\ref{def.3ebad}).
\end{theorem}

\begin{proof}
The second part of the theorem follows from Figure \ref{fig.bad}.

In order to prove the first part, without loss of generality we distinguish the following three cases. 

\begin{itemize}
 \item[(i)]  There are no $P$-no-edges of $K$ incident with exactly one vertex of $V(P_K)$;
that is, all $P$-no-edges of $K$ have both end vertices in $V(P_K)$. Then, the result of Theorem~\ref{thm.frsst} 
follows from Observation~\ref{o.triangle}.
\item[(ii)] Exactly four $P$-no-edges of $K$ are incident with exactly one vertex of $V(P_K)$. 
Hence, there is exactly one derived white edge with both end vertices in $V(P_K)$; let us denote it by~$e$.
  Edge $e$ connects either the two non-central hexagons $h_x$ and $h_y$, 
  or the central hexagon with one non-central hexagon; namely, $h_x$ and $h_z$ or $h_y$ and $h_z$.
Without loss of generality, these two cases are described in Figures~\ref{fig.41} and~\ref{fig.42}, where $e=e_1=e_2$. 
In both cases, the perfect matchings that lead to a safe reduction of $\{h_x, h_y, h_z\}$ 
  are indicated by thicker lines.
\item[(iii)] 
Exactly two
$P$-no-edges of $K$ are incident with exactly one vertex of $V(P_K)$. 
Hence, there are exactly two P-no-edges of $K$ with both end vertices in $V(P_K)$. 
Without loss of generality, the two cases that arise when they both connect vertices between the same pair 
of hexagons are depicted in Figures~\ref{fig.tre1} and \ref{fig.tre2}. 
If they connect vertices of distinct 
pairs of hexagons then aside of the $P$-bad ($P$-danger in case that there is no red-connection) there is one more 
case, which is depicted in Figure \ref{fig.tre3}. 
In all the three described cases, the perfect matchings that lead 
to a safe reduction of $\{h_x, h_y, h_z\}$ are indicated by thicker lines. 
\end{itemize}
\begin{figure}[h]
\centering
\subfigure[]
{
\ifpdf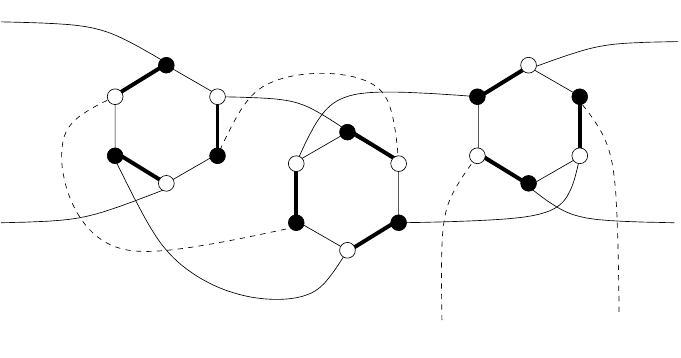\else\input{tre1.ps_t}\fi
 \label{fig.tre1}
}\qquad
\subfigure[]
{
\ifpdf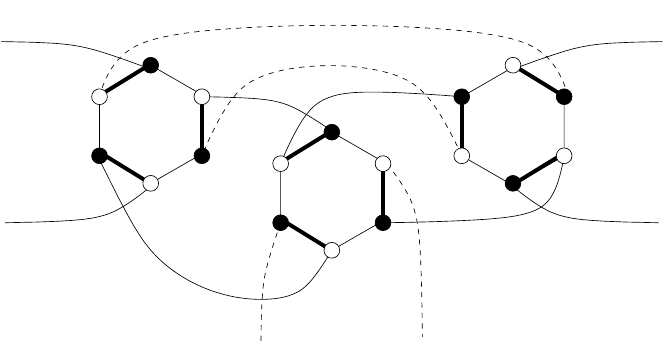\else\input{tre2.ps_t}\fi
 \label{fig.tre2}
}

\subfigure[]
{
\ifpdf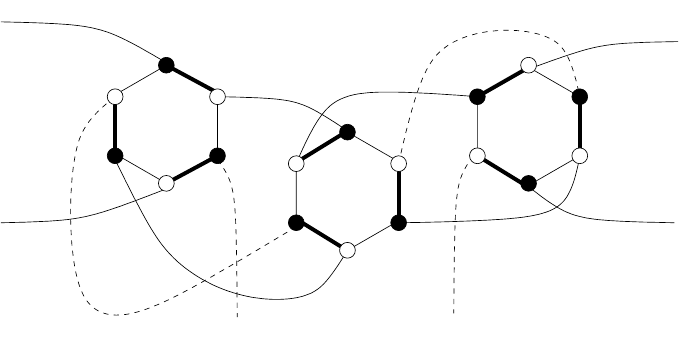\else\input{tre3.ps_t}\fi
 \label{fig.tre3}
}
\caption{The distinct configurations of the case (iii) of the proof of Theorem~\ref{thm.frsst}:
exactly two $P$-no-edges of $K$ have both end vertices in $V(P_K)$. 
Again, $P$-no-edges are represented by dotted lines.}
\end{figure}
\end{proof}

We define the remaining two obstacles for a correct reduction of a fork.

\begin{definition}
\label{def.abbad}
Let $K$ and $K'$ be proper pseudohexes such that $K$ 
is obtained from $K'$ by $F$-addition, where $F$ is the fork. 
We say that $K$ has a {\em $F$-abad} if $F_K$ is configured as in Figure~\ref{fig.abad}
and that $K$ has a {\em $F$-bbad} if $F_K$ is configured as in Figure~\ref{fig.bbad}. 
\end{definition}

\begin{figure}[h]
\centering
\subfigure[$F$-abad pseudohex]
{
\ifpdf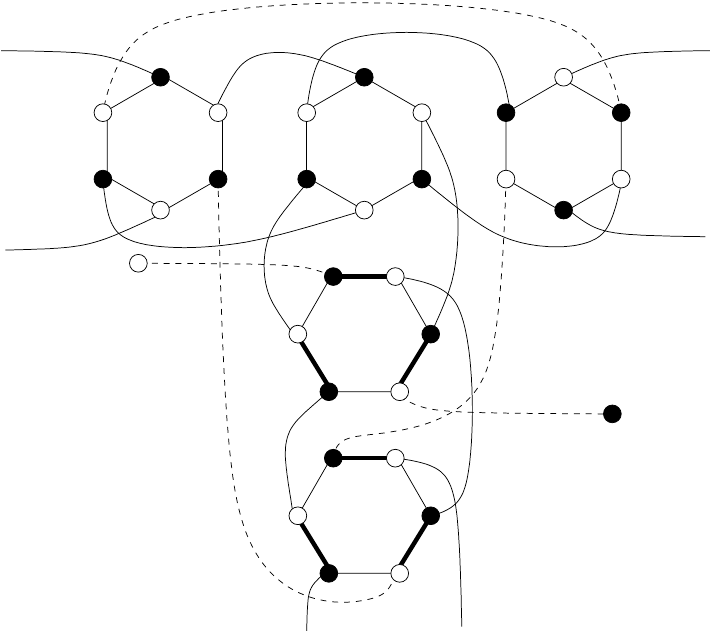\else\input{abad.ps_t}\fi
 \label{fig.abad}
}
\subfigure[$F$-bbad pseudohex]
{
\ifpdf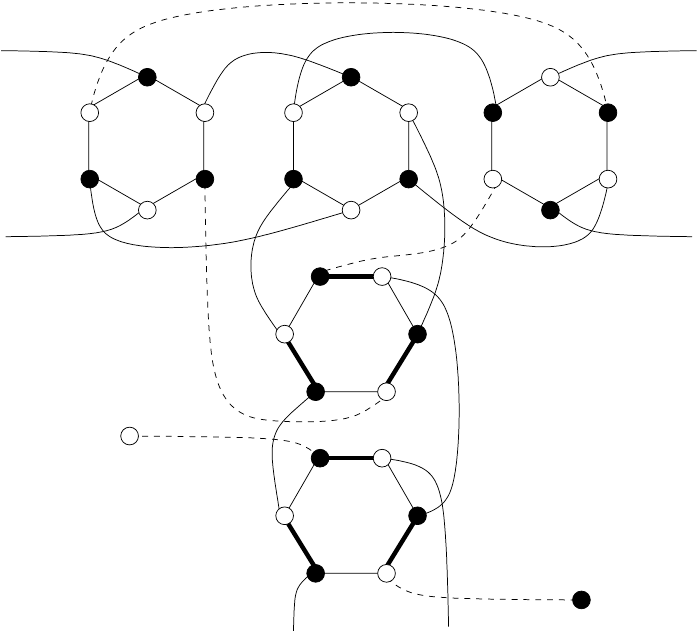\else\input{bbad.ps_t}\fi
 \label{fig.bbad}
}
\caption{$F$-no-edges are represented by dotted lines.}
\label{fig.abbads}
\end{figure}

\begin{theorem}
\label{thm.scnd}
Let $K, K'$ be proper pseudohexes such that $K$ is obtained from $K'$ by $F$-addition, where $F$ is the fork. 
If $K$ has neither a $F$-obstacle nor a $F$-abad nor a $F$-bbad, then 
there is a safe reduction of~$F_K$. 

Moreover, if $K$ has a $F$-abad or a $F$-bbad, and the edge $e$ (see Figure~\ref{fig.abbads}) 
is not red-connected to another $F$-no-edge $e'$ of $K$,
then there is a reduction of $F_K$ so that all but one contracted path have at most one derived white edge
and the contracted path that has more than
one derived white edge contains exactly two derived white edges, namely $e$ and $e'$.
\end{theorem}

Before presenting the proof of Theorem~\ref{thm.scnd}, 
we show that the following statement is a straightforward consequence of Theorem~\ref{thm.scnd}.


\begin{theorem}
\label{thm.fearr}
Let $K, K'$ be proper pseudohexes such that $K$ is obtained from $K'$ by $F$-addition, where $F$ is a fork.
If each $F$-no-edge of $K$ is a subset of $V(F_K)$, then 
there is a safe reduction of $F_K$. 
\end{theorem}

\begin{proof} 
By Theorem \ref{thm.scnd}, it is enough to prove that $K$ has neither
a $F$-obstacle, nor a $F$-abad and nor a $F$-bbad. 
We consider names of vertices, edges and hexagons of $F_K$ from Figure~\ref{fig.abbads}.
By the hypothesis, it is trivial the fact that  $K$ does not have a $F$-obstacle. 

For the sake of contradiction, we assume that $K$ has a $F$-abad ($F$-bbad, respectively).
Then $F_K$ is configured as in a $F$-abad ($F$-bbad, respectively). But since $K$ has no ends (proper), 
the edges with end vertices $r'_1, r'_2$ ($r_1, r_2$, respectively) are distinct and
therefore, each of them has an end-vertex in $V(K)-V(F_K)$, a contradiction.
The result follows.
\end{proof}




\subsubsection{Proof of Theorem~\ref{thm.scnd}}
\label{sub.prscnd}
In the following, we recall that we consider names of vertices, edges and hexagons of $F_K$ from Figure~\ref{fig.abbads}.
Also, using Figure~\ref{fig.abad}, we denote by $M_a$ and $M_b$ the perfect matchings 
indicated by thicker lines on hexagons $h_a$ and $h_b$, respectively. Moreover,
let $N_a$ and $N_b$ denote the perfect matchings of $h_a$ and $h_b$
that are complements of $M_a$ and $M_b$, respectively.

We first prove the second part of Theorem~\ref{thm.scnd}. 
If $K$ has a $F$-abad and $e$, $e'{=} \{r'_2,\hat{r}_2\}$ are not red-connected, then 
we reduce $h_a, h_b$  by $N_a, M_b$. 
If $K$ has a $F$-abad and $e$, $e'{=}\{r'_1,\hat{r}_1\}$ are not 
red-connected, then we reduce $h_a, h_b$  by $M_a, N_b$. 
If $K$ has a $F$-bbad and $e$, $e'{=}\{r_2,\hat{r}_2\}$ are not red-connected, 
then we reduce $h_a, h_b$  by $N_a, M_b$. If $K$ has a $F$-bbad and $e$, 
$e'{=}\{r_1,\hat{r}_1\}$ are not red-connected, then we reduce $h_a, h_b$  by $M_a, N_b$. 
It is a routine to check that all these reductions of $\{h_a, h_b\}$ are safe, 
and that in the resulting pseudohex a $P$-danger is generated, where $P$ is the 3-ear with 
vertex set $\{x, z, y\}$, but the $P$-danger is not a $P$-bad; since $e$
is not red-connected to the $P$-no-edge 
obtained from contracting a path that contains $e'$, and $e'$ is incident to a vertex in $V(h_z)$
and to a vertex in $V(h_x {\cup} h_y {\cup} h_z)$. 
Hence, the second statement of Theorem~\ref{thm.scnd}  
follows from the second part of Theorem~\ref{thm.frsst}. 

In order to prove the first part of the statement 
of Theorem~\ref{thm.scnd}, we assume that $K$ has neither a $F$-obstacle, nor a $F$-abad, nor a $F$-bbad.

Without loss of generality, we distinguish three main cases; 
Case~1: $\{r_1,r'_1\}$  is a white edge of $K$, 
Case~2: $\{r_2, r'_2\}$ is a white edge of $K$, 
and Case~3: Neither $\{r_1,r'_1\}$ nor $\{r_2, r'_2\}$ are white edges of $K$. 
In what follows, we deal with the analysis of these cases.

{\bf Case 1:} If $\{r_1,r'_1\}$ is a white edge of $K$, then 
we reduce $h_a, h_b$ by $M_a, M_b$.

{\bf Case 2:} If $\{r_2, r'_2\}$ is a white edge of $K$, then  
we reduce $h_a, h_b$ by $N_a, N_b$.

Clearly, the reductions indicated in Cases 1 and 2 are safe.

Further, in both cases a path consisting entirely of real white edges is contracted to a 
single new derived white edge; let us denote it by $e_r$.
In Case 1, $e_r$ is incident to vertex $w$ and in Case 2, $e_r$ is incident to vertex $w'$. 

We assume that Case 1 holds (Case 2 is discussed analogously).
If reduction of $h_a$, $h_b$ by $M_a$, $M_b$ generates
a $P$-danger, where $P$ is the 3-ear with vertex set $\{x, z, y\}$, then necessarily $K$ has a $F$-no-edge 
with its end vertices in $V(h_x \cup h_y)$. We may assume it is the edge
$e$ of Figure \ref{fig.bad}. 
However, this edge $e$ is not red-connected to the new derived white edge $e_r$; which has been
producing by contracting a path that consists only of real
white edges of $K$. 
Hence, we have a $P$-danger
but not $P$-bad and the result follows from the second part of Theorem~\ref{thm.frsst}. 

Therefore, we need to assume that the indicated reduction do not generate a $P$-danger.
If the indicated reductions do not generate a $P$-obstacle then we are done by Theorem~\ref{thm.frsst}.
Hence, we assume that a $P$-obstacle is generated.
Recall that there exist one $P$-no-edge, namely $e_r$, that is the result
of the contraction of a path that consists only of real
white edges of $K$. Therefore, by Observation~\ref{o.ob2}, to have the desired result 
it is enough to prove that there is a potential pair of $P$-no-edges 
containing $e_r$ that is not red-connected. 

As we are in Case 1, the edge $e_r$ is incident to $w$ which belongs
to the bipartition class represented by black vertices in Figure \ref{fig.abad}. 
For the sake of contradiction we assume that $e_r$ 
is red-connected to both $P$-no-edges incident to two available 
white vertices, say $x'$ and $y'$, in $V(h_x)$ and $V(h_y)$, respectively. Then, each 
of these two $P$-no-edges must be created by contracting paths that intersect at 
least one of the hexagons $h_a, h_b$. Hence, in $K$ there are white edges connecting 
$x',y'$ to $r'_1, r_2$, which contradicts the assumption of Case 1 that $\{r_1,r'_1\}$ is a white edge of $K$.

\medskip\noindent

{\bf Case 3:} The vertices $r_1, r'_1, r_2, r'_2$ are not connected by white edges of $K$. 
	      We note that reductions of $h_a, h_b$ by $M_a, N_b$ and by $N_a, M_b$ are safe. 
	      Without loss of generality, we consider two subcases.

	      \medskip\noindent

{\bf Case 3.1:} None of the four distinct $F$-no-edges incident to a vertex in
$\{r_1, r'_1, r_2, r'_2\}$ is incident to a vertex in $V(h_x \cup h_z \cup h_y)$.

Since $K$ is not a $F$-obstacle, at least one 
  $F$-no-edge of $K$ has its end vertices in $V(h_x \cup h_z \cup h_y)$.
In this case no reduction of $h_a, h_b$ can lead to a $P$-obstacle. Moreover, if 
in the pseudohex obtained by reducing $h_a, h_b$ by $M_a, N_b$ 
both new derived white edges incident to $w, w'$ are not incident to a vertex in $V(h_x \cup h_z \cup h_y)$ 
and thus this reduction does not lead to a $P$-danger either. 
Hence, the statement of Theorem~\ref{thm.scnd} follows from Theorem~\ref{thm.frsst}.

\medskip\noindent

{\bf Case 3.2.} At least one of the four $F$-no-edges incident to 
a vertex in $\{r_1, r'_1, r_2, r'_2\}$ is also incident to a vertex in $V(h_x \cup h_z \cup h_y)$.

{\bf Case 3.2.1.} For each $i= 1,2$, exactly one of the $F$-no-edges 
incident to a vertex in $\{r_i, r'_i\}$ is incident to a vertex in $V(h_x \cup h_z\cup h_y)$. 
Hence, each of the two reductions of $h_a, h_b$ by $M_a,N_b$ and $M_b, N_a$  is safe
and creates a derived white edge that has both end vertices either in $V(h_x \cup h_z)$
or in $V(h_y \cup h_z)$
and thus, in particular, none of these two reductions creates a $P$-obstacle, where $P$ is the 3-ear on vertex set $\{x,y,z\}$.
By Theorem~\ref{thm.frsst}, it is enough to show that for each possible
case, at least one of these reductions does not generate a $P$-danger.

For the sake of contradiction, suppose that for some of the cases
both reductions generate a $P$-danger.
As we observed reductions of $h_a, h_b$ by $M_a,N_b$ and $M_b, N_a$
create derived white edge with both end vertices either in $V(h_x \cup h_z)$
or in $V(h_y \cup h_z)$. As both reductions generate a $P$-danger, 
it implies that there exist a $F$-no-edge with its end-vertices in $V(h_x \cup h_y)$.
Without loss of generality, we may assume it is the edge $e$ of Figure \ref{fig.abad}.

Hence the two $F$-no-edges incident to a vertex in $\{r_1, r'_1, r_2, r'_2\}$ 
and to a vertex in $V(h_x \cup h_z \cup h_y)$
are incident to vertices $t, t'$ in $V(h_x \cup h_z \cup h_y)$. These two vertices 
belong to different bipartition classes of $K$ 
and so the $F$-no-edges incident to them are incident either to $r_1, r_2$, or to $r'_1, r'_2$. 
Hence, $F_K$ is configured as a $F$-abad or a $F$-bbad, see Definition \ref{def.abbad}.
A contraction to the hypothesis assumption.

{\bf Case 3.2.2.} Suppose that Case 3.2.1. does not hold.
If both $F$-no-edges incident to $r_1, r'_1$ are incident to vertices in $V(h_x\cup h_z\cup h_y)$ 
then reduction of $h_a, h_b$ by $M_a, N_b$ is safe and generates two distinct derived
white edges, each of them with an end vertex in $\{w, w'\}$ and an end-vertex in $V(h_x \cup h_z \cup h_y)$. 
Thus, neither a $P$-obstacle nor
a $P$-danger is created. Analogously we solve the case that the two $F$-no-edges 
incident to $r_2, r'_2$ are incident to $V(h_x \cup h_z \cup h_y)$.

Without loss of generality, we are left with the cases that there is exactly one $F$-no-edge with
an end-vertex in $\{r_1, r'_1\}$, say $r_1$, and one end-vertex in $V(h_x \cup h_z \cup h_y)$ and 
there are no $F$-no-edges incident to both subsets, $\{r_2, r'_2\}$ and $V(h_x \cup h_z \cup h_y)$. 

By Theorem~\ref{thm.frsst}, it is enough to prove that there is a safe reduction of $\{h_a, h_b\}$
that does not generate either a $P$-obstacle, or a $P$-danger. 
Since, we have exactly three available vertices in $V(h_x \cup h_y)$
which are end-vertices of $F$-no-edges and are not incident to $V(h_a \cup h_b)$,
the following two cases arise.
   
{\bf Case 3.2.2.2.} Three $F$-no-edges incident to $V(h_x \cup h_z \cup h_y)$ 
have exactly one end vertex in $V(F_K)$. Then no reduction of 
$h_a, h_b$ can create a $P$-danger. 
Moreover, by assumption, there is a $F$-no-edge incident to $r_1$ and to 
a vertex in $V(h_x\cup h_z \cup h_y)$, and thus, the safe reduction of $h_a, h_b$ by $M_a,N_b$ creates a $P$-no-edge 
with end vertices in $V(h_x \cup h_z \cup h_y)$ and hence this reduction does not 
lead to a $P$-obstacle. Consequently, in this case, Theorem~\ref{thm.scnd} holds.

{\bf Case 3.2.2.1.} One $F$-no-edge has both end vertices in $V(h_x\cup h_z\cup h_y)$. 
 Then, no reduction of $h_a, h_b$ can create a $P$-obstacle. 
 Furthermore, the safe reduction of $h_a,h_b$ by $N_a,M_b$  creates derived white 
 edges incident to $w, w'$ and to  $V(h_x\cup h_z \cup h_y)$ and then, this 
 reduction does not create a $P$-danger and hence, Theorem \ref{thm.scnd} holds.


\subsection{Reduction of star forks, subforks and dots on pseudohexes} 
\label{sec.subforks_subears}

In this section we prove the following theorem.

\begin{theorem}\label{theo.sub-fork-ear}
Let $L$ be the star fork, the subfork, or the dot. Let $K, K'$ be proper pseudohexes and 
$K$ obtained from $K'$ by $L$-addition. Then $L_K$ 
can be safely reduced. 
\end{theorem}

\begin{proof}
The proof in the case that $L$ is the star fork is contained in Subsection~\ref{sub.supp}. 
If $L$ is the dot, then it is trivial that every reduction of $L_K$ is safe.

Let $L$ be the subfork. 
In the following, we shall use the names of vertices and hexagons of $L_K$ from Figure~\ref{fig.subear-1-proof}. 
 Two subcases arise: (2.1) at least one $L$-no-edge has both end-vertices
 in $V(L_K)$ and (2.2) all $L$-no-edges have
 exactly one end-vertex in $V(L_K)$. Let $M_z$ and $M_b$ denote the perfect matchings depicted
 by thicker lines in Figure~\ref{fig.subear-1-proof} on $h_z$ and $h_b$, respectively.
 Moreover, let there $N_z$ and $N_b$ denote the perfect matchings of $h_z$ and $h_b$ 
 that are complements of $M_z$ and $M_b$, respectively. 

If $K$ is such that subcase (2.1) holds, and without loss of
generality $\{t,w\}$ (see Figure~\ref{fig.subear-1-proof} for notation)
 is 
 a derived white edge, then reduction of $h_z, h_b$ by $N_z, M_b$ is safe.
In the subcase (2.2) the reduction of $h_z, h_b$ by $M_z, M_b$ is safe.


 \begin{figure}[h]
\centering
\ifpdf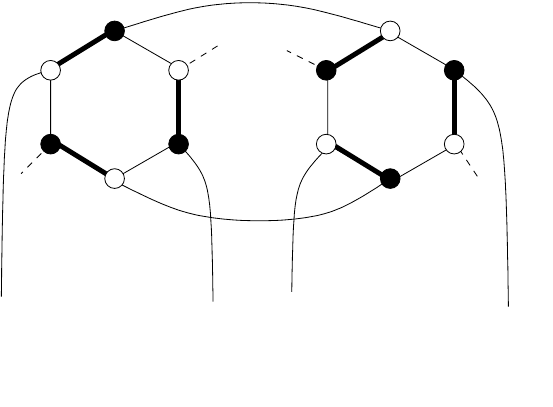\else\input{subear-1-proof.ps_t}\fi
\caption{Subfork pseudohex. Perfect matchings $M_z$ and $M_b$ depicted in $h_z$ and $h_b$ by thicker lines.
Derived white edges are represented by dotted lines.}
\label{fig.subear-1-proof}
\end{figure}


\end{proof}

\subsection{Reduction of p-forks on pseudohexes}
 \begin{figure}[h]
\centering
\ifpdf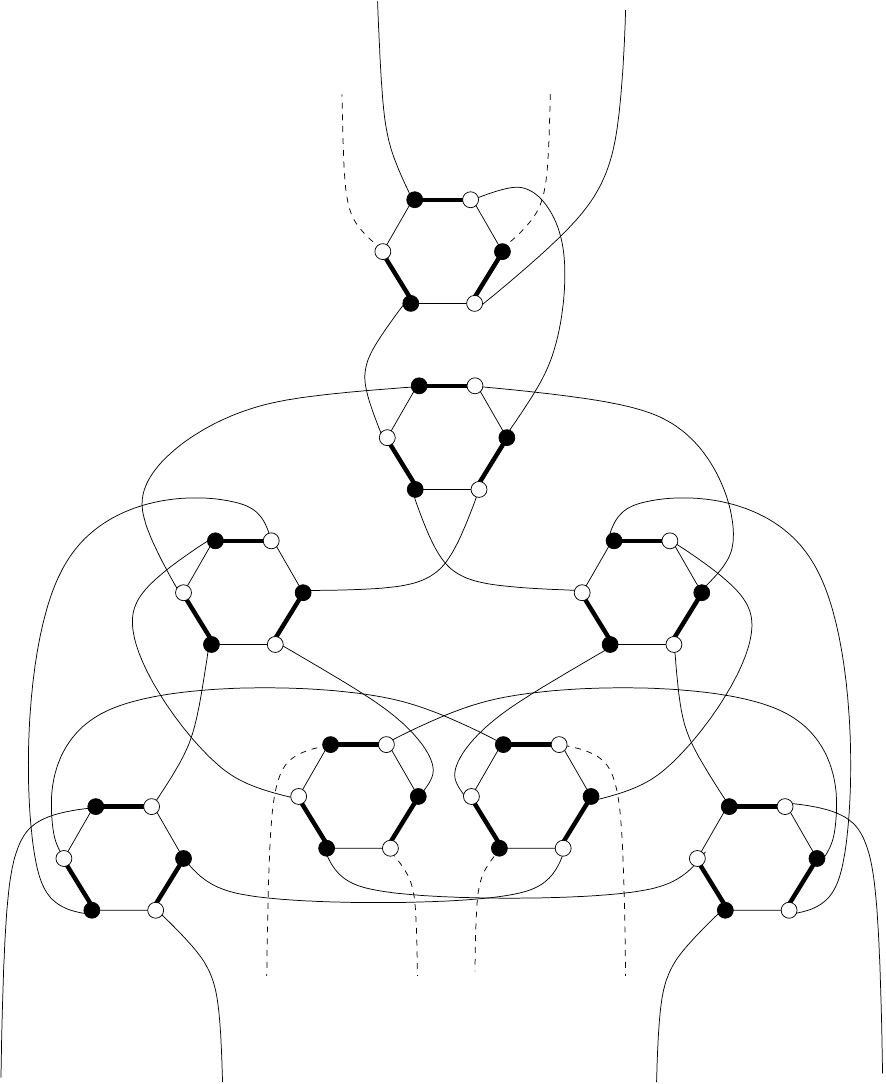\else\input{petfork.ps_t}\fi
\caption{p-fork Pseudohex. Perfect matchings $M_a$, $M_b$, $M_a'$, $M_b'$, $M_z$, $M_z'$, $M_x$, $M_y$ 
depicted in $h_a$, $h_b$, $h_a'$, $h_b'$, $h_z$, $h_z'$, $h_x$, $h_y$ by thicker lines.
Derived white edges are dotted.}
\label{fig.petforkpse}
\end{figure}

We prove that there is a safe reduction of the p-fork on pseudohexes.

\begin{theorem}\label{theo.p-fork}
Let $L$ be a p-fork. 
Let $K, K'$ be proper pseudohexes and 
$K$ obtained from $K'$ by $L$-addition.
Then $L_K$ has a safe reduction. 
\end{theorem}
\begin{proof}
We consider the notation from Figure~\ref{fig.petforkpse}. \\
In the case that the derived edges from the set $\{b_{x}, w_{x}, b_{y}, w_{y}, b_{z}, w_{z}\}$ are pairwise distinct,
the reduction of $L_K$ by 
$N_z$, $N_z'$, $N_b'$, $N_a'$, $N_b$, $M_a$, $M_x$, $M_y$ is safe.
We suppose that there are exactly two derived white edges that are the same, but the rest are all pairwise
distinct. Without loss of generality, two cases arise: (1) $b_z=w_x$, and (2) $b_x=w_y$.
In the first case, we reduce $L_K$ by 
$N_z$, $N_z'$, $M_b'$, $N_a'$, $M_b$, $N_a$, $N_x$, $N_y$, finding a safe reduction
and in the second one by $M_z$, $N_z'$, $N_b'$, $M_a'$, $N_b$, $M_a$, $M_x$, $M_y$.
If exactly two pairs of derived white edges are the same, again without loss of generality we have two cases:
(1')  $w_x=b_y$ and $w_y = b_z$,
and (2') $w_z=b_x$ and $b_z=w_y$. 
The reductions of $L_K$ by $N_z$, $M_z'$, $N_b'$, $M_a'$, $N_b$, $M_a$, $M_x$, $M_y$ for case (1') 
and by $N_z$, $M_z'$, $M_b'$, $N_a'$, $N_b$, $M_a$, $M_x$, $M_y$ for case (2') are safe reductions.
We are left with the case that all derived white edges have both end-vertices in $L_K$.
Then, the reduction of $L_K$ by the perfect matchings $M_z$, $N_z'$, $N_b'$, 
$N_a'$, $N_b$, $M_a$, $N_x$, $M_y$ is safe.
\end{proof}

In the next Subsection we study safe reductions
of $j$-big-forks on pseudohexes.

\subsection{Reduction of $j$-big-forks on pseudohexes}
\label{sub:bigfork}

We first need to introduce the concept of $B$-obstacles in the case that $B$ a $j$-big-fork. This notion
is equivalent to the notion of cut-obstacles. 

\begin{definition}
Let $K, K'$ be proper pseudohexes such that $K$ is obtained from 
$K'$ by $B$-addition, where $B$ is a $j$-big-fork. 
We say that $K$ has a {\em $B$-obstacle} if it has $2(4+j)$ distinct
$B$-no-edges.  
\end{definition}

In Theorem~\ref{thm.superfork}, we consider the case $j=1$. 
We recall that 1-big-forks are simply called big-forks. 

\begin{theorem}
\label{thm.superfork} 
Let $K, K'$ be proper pseudohexes such that $K$ is obtained from 
$K'$ by $B$-addition, where $B$ is the big-fork. 
If $K$ does not have a $B$-obstacle, then there is a safe reduction of $B_K$.
\end{theorem}

\subsubsection{Proof of Theorem \ref{thm.superfork}}
\label{sub.supp}
Let $F$ and $T$ denote the fork and the star fork respectively used to obtain $B$ 
(as explained in Figure~\ref{fig.bigfork}).

By Theorem~\ref{thm.scnd} it suffices to show that there is a safe reduction of $T_K$
that generates a pseudohex that has neither a $F$-obstacle, nor a $F$-abad, nor a $F$-bbad. 
We say that such a reduction of $T_K$ is \emph{fundamental}.

In this proof we use the names of vertices of the big-fork from Figure~\ref{fig.bigfork}. 
Namely, the vertex set of $T$ is $\{x',y',z',b'\}$ and the vertex set of $F$ is $\{x,y,z,a,b\}$.
In addition, we consider the bipartition classes of $K$ as white and black vertices.
We further denote by $b_{x'}$ the derived white edge incident to a black vertex of $h_{x'}$
and by $w_{x'}$ the derived white edge incident to a white vertex of $h_{x'}$. Analogously,
we use this notation for the derived white edges incident to $h_{y'}, h_{b'}$.
Let $M_{x'}, M_{y'}, M_{z'}, M_{b'}$ be the perfect matchings 
indicated in Figure~\ref{fig.star_fork_all} by thicker lines,
and we denote by $N_{x'}, N_{y'}, N_{z'}$ and $N_{b'}$ 
the perfect matchings that are complement
of $M_{x'}, M_{y'}, M_{z'}$ and $M_{b'}$, respectively.

\begin{figure}[h]
\centering
\ifpdf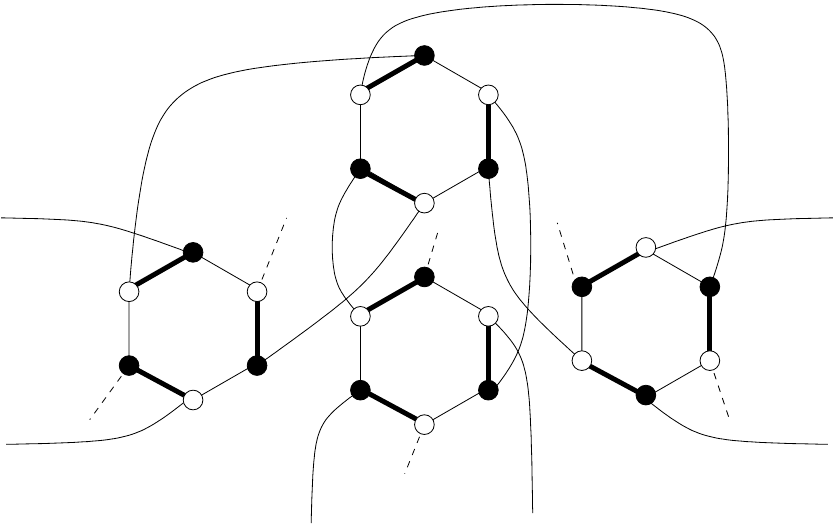\else\input{star_fork_all.ps_t}\fi
\caption{Star fork. Perfect matchings $M_{x'}$, $M_{y'}$, $M_{z'}$ and $M_{b'}$ 
are depicted in $h_{x'}$, $h_{y'}$, $h_{z'}$ and $h_{b'}$ by thicker lines. Derived white edges are dotted}
\label{fig.star_fork_all}
\end{figure}

We still introduce one more notation. The hexagons $h_{x'}, h_{y'}, h_{b'}$ are connected by three pairs
of real white edges to three pairs of vertices in $V(K)\setminus V(T_K)$. Let us denote by $u_1, v_1$ 
the pair connected to $h_{x'}$, by  $u_2, v_2$ the pair connected 
to $h_{y'}$ and by $u_3, v_3$ the pair connected 
to $h_{b'}$. 
We note that for each $i=1,2,3$, the vertices $u_i, v_i$ belong to the same hexagon of $K$. 
We further note that $u_1, v_1, u_2, v_2$ are in $V(F_K)$ 
and the $u_3, v_3$ belong to $V(K')$. 
Let us assume that $u_i$ ($v_i$, respectively) belongs to the bipartition class of black
vertices (white vertices) of $K$, for each $i=1,2,3$.

By the hypothesis assumption, there exists a $B$-no-edge with both end-vertices in $V(B_K)$. 
We distinguish two cases.

{\bf Case 1.} There are distinct $p,q \in \{x',y',b'\}$
such that $w_{p}=b_{q}$ and $b_{p}=w_{q}$. Let $r \in \{x',y',b'\}-\{p,q\}$. 
Then both edges $w_{r}, b_{r}$ have exactly one end-vertex in $V(T_K)$.
We safely reduce hexagons $h_{x'}, h_{y'}, h_{z'}, h_{b'}$ by $M_{x'}, M_{y'}, M_{z'}, M_{b'}$. 
Two cases may happen: $w_{r}, b_{r}$ have end-vertices either in $V(F_K)$, or in $V(K')$. 
We note that in both cases the suggested reduction creates one derived white edge $e'$ 
subset of $V(F_K)$, and another derived white edge $e''$ incident to the 
same hexagon of $F_K$ as $e'$ that is incident to $V(F_K)$ and to $V(K')$. 
Hence, this reduction is fundamental.

{\bf Case 2.} Case 1 does not happen. Then we can assume without loss of generality that the 
$B$-no-edges which are subsets of $V(T_K)$ belong to the set $\{w_{x'}=b_{y'}, w_{y'}=b_{b'}, w_{b'}=b_{x'}\}$.

We first make an observation that
follows directly from the assumption of Case 2, Figure \ref{fig.star_fork_all}
and the symmetry of the star fork. 
\begin{observation}
\label{o.super}
The reduction of $h_{x'}, h_{y'}, h_{z'}, h_{b'}$ by $N_{x'}, M_{y'}, N_{z'}, N_{b'}$
is safe and generates a pseudohex so that the 
$F$-no-edge incident to $u_1$ ($u_2$, $u_3$, $v_1$, $v_2$ and $v_3$, respectively) is obtained
by contracting a path that contains $b_{b'}$ ($b_{x'}$, $b_{y'}$, $w_{x'}$, $w_{y'}$ and $w_{b'}$, respectively).
By symmetry of the star-fork, there are safe reductions of $T_K$
such that the $F$-no-edge incident to $u_1$ ($u_2$, $u_3$, $v_1$, $v_2$ and $v_3$, respectively) is obtained
by contracting a path that contains 
$b_{y'}$ ($b_{b'}$, $b_{x'}$, $w_{x'}$, $w_{y'}$ and $w_{b'}$, respectively), 
 $b_{x'}$ ($b_{y'}$, $b_{b'}$, $w_{b'}$, $w_{x'}$ and $w_{y'}$, respectively),
and $b_{x'}$ ($b_{y'}$, $b_{b'}$, $w_{y'}$, $w_{b'}$ and $w_{x'}$, respectively).
\end{observation}


We split up Case 2 into subcases.

{\bf Case 2.1}  $w_{x'}=b_{y'}$, $w_{y'}=b_{b'}$, $w_{b'}=b_{x'}$. Then, we reduce 
$h_{x'}, h_{y'}, h_{z'}, h_{b'}$ by $N_{x'}, M_{y'}, N_{z'}, N_{b'}$.
By Observation~\ref{o.super}, this is a safe reduction and 
it creates derived white edges $\{u_1, v_2\}$, $\{v_1, u_3\}$ and  $\{u_2, v_3\}$.
Since, $\{u_1, v_2\}$ is a subset of $V(F_K)$ and each $\{v_1, u_3\}$,  $\{u_2, v_3\}$
has exactly one end-vertex in $V(F_K)$ and these end-vertices belong to distinct
hexagons of $F_K$, this reduction is fundamental.

{\bf Case 2.2}  There are distinct $p, q, r \in \{x',y',b'\}$ so that $w_p= b_q$ and $w_r= b_p$. 
We first assume that either $w_{x'}= b_{y'}$ and $ w_{y'}= b_{b'}$, or $w_{y'}=b_{b'}$ and $w_{b'} = b_{x'} $. 
By Observation~\ref{o.super}, the reduction of
$h_{x'}, h_{y'}, h_{z'}, h_{b'}$ by $N_{x'}, M_{y'}, N_{z'}, N_{b'}$ 
creates derived white edges either $\{v_1, u_3\}$ and  $\{v_2, u_1\}$, or
$\{u_2, v_3\}$ and  $\{u_1, v_2\}$.
We now assume that $w_x=b_y$ and $w_b=b_x$. 
By Observation~\ref{o.super}, there is a the reduction of
$h_{x'}, h_{y'}, h_{z'}, h_{b'}$ that
creates derived white edges $\{u_2, v_3\}$ and  $\{u_1,v_2\}$.
Therefore, in all the cases there is a $F$-no-edge $e$ with both end-vertices in $V(F_K)$
and a $F$-no-edge $e'$ with exactly one end-vertex in $V(F_K)$ such that
$e$ and $e'$ are incident to the same hexagon in $F_K$. 
Hence, for each case the indicated reduction is fundamental.




{\bf Case 2.3}  There are distinct $p, q \in \{x',y',b'\}$ so that $w_p= b_q$.  

If at least three derived white edges incident to vertices in $V(F_K)$ have an
end vertex in $V(K')$ then, by Observation~\ref{o.super} 
 there exists a safe reduction of $T_K$ that creates a $F$-no-edge with both
 end-vertices in $V(F_K)$ (such a $F$-no-edge is the one created by the contraction
 of the path that contains $w_p= b_q$). Moreover, the resulting pseudohex
 still has at least three derived white edges incident to vertices in $V(F_K)$ and each with
 an end-vertex in $V(K')$. Hence, this reduction is fundamental.

If $\{u_1, u_2, v_1, v_2\}\cap V(F_K)\subset V(h_a\cup h_b)$ then, by Observation~\ref{o.super} 
there exists a safe  reduction of $T_K$ that creates a $F$-no-edge with both
 end-vertices in $V(h_a \cup h_b)$. Then, this reduction is fundamental; 
 see Figures~\ref{fig.abad},~\ref{fig.bbad}.

The remaining case is that $\{u_1, u_2, v_1, v_2\}\cap V(h_x\cup h_y)\neq \emptyset$ and 
at most two derived white edges with an end-vertex in $V(F_K)$ have end-vertices in $V(K)$.
By Observation~\ref{o.super} there exists a reduction of $T_K$ that creates a derived white edge 
(denoted say by $e'$ and obtained by contracting the path that contains $w_{p}= b_{q}$) 
with an end-vertex in $V(h_x\cup h_y)$ and an end-vertex in $V(K')$.  
Therefore, this reduction generates neither a $F$-bbad nor a $F$-abad, and
hence, we only need to show that this reduction does not create a $F$-obstacle.

If $V(F_K)$ has a derived white edge as a subset then the last is trivial.
Hence, let us assume that $V(F_K)$ have no derived white edges as a subset.
There are $4$ vertices of $V(F_K)$ contained in $B$-no-edges of $K$, 
and by assumption, at least two of them are in a $B$-no-edge with end-vertices in $V(T_K)$. 
One real white edge with end-vertices in $V(T_K)$ and in $V(K')$ is in the contracted
path that creates $e'$, and hence at least one 
derived white edge with end-vertices in $V(T_K)$ and in $V(F_K)$ 
is not on a contracted path that creates a new derived white edge incident to a 
vertex in $V(K')$. Then, this reduction creates a derived white edge with both end-vertices
in $V(F_K)$, and we thus, a $F$-obstacle is not generated. 

{\bf Case 2.4}  All derived white edges incident to a vertex in $V(T_K)$ 
are incident to a vertex in $V(K)\setminus V(T_K)$. 
We recall that $K$ is not a $B$-obstacle, and so at least one derived white 
edge, say edge $e'$,  has end-vertices in $V(T_K)$ and in $V(F_K)$.
Now we proceed similarly as in Case 3. 
There are $6$ distinct $B$-no-edges with one end-vertex in $V(T_K)$, and at most 
$4$ of them have an end-vertex in $V(F_K)$. Hence at least two of them have an end-vertex in $V(K')$.

If $\{u_1, u_2, v_1, v_2\}\cap V(h_x\cup h_y)\neq \emptyset$ then by Observation~\ref{o.super}
there exists a safe reduction of $T_K$ so that in the resulting pseudohex there is a new derived
white edge with end-vertices in $V(h_x\cup h_y)$ and in $V(K')$, and also a new derived white edge 
that has both end-vertices in $V(F_K)$. 
Hence, this reduction is fundamental, see Figures~\ref{fig.abad},~\ref{fig.bbad}.

Hence, let $\{u_1, u_2, v_1, v_2\}\cap V(F_K)\subset V(h_a\cup h_b)$. 
First let there be a $B$-no-edge with both end-vertices in $V(F_K)$. 
If there are two  $B$-no-edges with end-vertices in $V(T_K)$ and in $V(K')$ and these
two edges are incident to different hexagons of $T_K$, then by 
Observation~\ref{o.super} there is a safe reduction of $T_K$ such that for both hexagons 
$h_a, h_b$ there are new derived white edge incident to them and to $V(K')$. This reduction is fundamental 
by Figures~\ref{fig.abad},~\ref{fig.bbad}. Otherwise, necessarily exactly two $B$-no-edges 
have end-vertices in $V(T_K)$ and in $V(K')$, and four $B$-no-edges have end-vertices in $V(F_K)$.
By Observation~\ref{o.super}, there is a  safe reduction of $T_K$ such that 
there are new derived white edges with both end-vertices in $V(F_K)$ and incident
to $V(h_a)$ and to $V(h_b)$ (not necessarily one derived white edge incident to both). 
Hence, this reduction is fundamental.

Let there be no $B$-no-edge subset of $V(F_K)$. If at least one $B$-no-edge have its end-vertices in
$V(F_K)$ and in $V(K')$, 
then we note that by our assumptions such $B$-no-edge is incident to $V(h_x\cup h_y)$.
By Observation~\ref{o.super}, there exists a safe reduction of $T_K$ so that the new derived 
white edge obtained by contracting the path that contains $e'$ is incident 
to $V(h_a\cup h_b)$ and thus it is a subset of $V(F_K)$.
Therefore, this reduction is fundamental. 

Finally let all four $B$-no-edges have their end-vertices in $V(h_x\cup h_y)$ and in $V(T_K)$. 
Then by Observation~\ref{o.super}, there exists  
a safe reduction of $T_K$ so that there are new derived white edges 
incident to $h_a$ and to $h_b$ which are subsets of $V(F_K)$. 
Consequently, this reduction is fundamental. 

\hfill $\square$
 
Finally, we extend Theorem~\ref{thm.superfork} from $j=1$ to general $j\geq 1$.     

\begin{theorem}
\label{thm.j-bigforks} 
Let $K, K'$ be proper pseudohexes such that $K$ is obtained from 
$K'$ by $B$-addition, where $B$ is the $j$-big-fork for $j\geq 1$. 
If $K$ does not have a $B$-obstacle,
then there is a safe reduction of~$B_K$.
\end{theorem}
\begin{proof}
By induction on $j$. By Theorem~\ref{thm.superfork}, the statement
holds for $j{=}1$. Let $j>1$ and $B$ be a $j$-big-fork. Let there
$B'$ and $T$ denote the $(j-1)$-big-fork and the star, respectively, from
which $B$ is obtained.
Since $K$ does not have a $B$-obstacle there is a derived white edge, say $e$, with both end vertices
in $V(B_K)$. If such a derived white edge has both end-vertices in $B'_K$, then by the induction
hypothesis and Theorem~\ref{theo.sub-fork-ear}, we can conclude that there is a safe reduction of $B_K$.
If $e$ has one end-vertex in $T_K$ and the other one in $B'_K$ or 
both end-vertices in $T_K$, by the Observation~\ref{o.super} contained in the proof of Theorem~\ref{thm.superfork}, 
we have that there is a safe reduction of $T_K$ that creates a derived white edge with both end-vertices
in $K_{B'}$ and therefore, again by the induction hypothesis we can find a safe reduction $B'_K$.
The result follows.
\end{proof}

\section{Acknowledgement}  We would like to thank to Mihyun Kang for extensive discussions.



\begin{thebibliography}{10}

\bibitem{AlonTarsi}
N.~Alon and M.~Tarsi.
\newblock Covering multigraphs by simple circuits.
\newblock {\em Society for Industrial and Applied Mathematics. Journal on
  Algebraic and Discrete Methods}, 6(3):345--350, 1985.

\bibitem{journals/dm/AlspachZ93}
B.~Alspach and C.-Q. Zhang.
\newblock Cycle covers of cubic multigraphs.
\newblock {\em Discrete Mathematics}, 111(1-3):11--17, 1993.

\bibitem{c}
D.~Cimasoni.
\newblock Discrete {D}irac operators on {R}iemann surfaces and {K}asteleyn
  matrices.
\newblock {\em Journal of the European Mathematical Society}, 14(4):1209--1244,
  2012.

\bibitem{Goddyn1989253}
L.~Goddyn.
\newblock Cycle double covers of graphs with hamilton paths.
\newblock {\em Journal of Combinatorial Theory, Series B}, 46(2):253 -- 254,
  1989.

\bibitem{Jaeger1979205}
F.~Jaeger.
\newblock Flows and generalized coloring theorems in graphs.
\newblock {\em Journal of Combinatorial Theory, Series B}, 26(2):205 -- 216,
  1979.

\bibitem{Jaeger19851}
F.~Jaeger.
\newblock A survey of the cycle double cover conjecture.
\newblock In B.R. Alspach and C.D. Godsil, editors, {\em Annals of Discrete
  Mathematics (27): Cycles in Graphs}, volume 115 of {\em North-Holland
  Mathematics Studies}, pages 1 -- 12. North-Holland, 1985.

\bibitem{JamshyTarsi}
U.~Jamshy and M.~Tarsi.
\newblock Short cycle covers and the cycle double cover conjecture.
\newblock {\em Journal of Combinatorial Theory. Series B}, 56(2):197--204,
  1992.

\bibitem{jkl}
A.~Jim\'enez, M.~Kang, and M.~Loebl.
\newblock Directed {C}ycle {D}ouble {C}overs: {H}exagon {G}raphs.
\newblock To appear in {\em The Seventh European Conference on Combinatorics, Graph
  Theory and Applications (EUROCOMB 2013)}, volume~16 of {\em Publications of the Scuola
  Normale Superiore}.

\bibitem{k}
R.~Kenyon.
\newblock The {L}aplacian and {D}irac operators on critical planar graphs.
\newblock {\em Inventiones Mathematicae}, 150(2):409--439, 2002.

\bibitem{MR904405}
L.~Lov{\'a}sz.
\newblock Matching structure and the matching lattice.
\newblock {\em Journal of Combinatorial Theory. Series B}, 43(2):187--222,
  1987.

\bibitem{m}
Ch. Mercat.
\newblock Discrete {R}iemann surfaces and the {I}sing model.
\newblock {\em Communications in Mathematical Physics}, 218(1):177--216, 2001.

\bibitem{opac-b1131924}
B.~Mohar and C.~Thomassen.
\newblock {\em Graphs on surfaces}.
\newblock Johns Hopkins studies in the mathematical sciences. Johns Hopkins
  University Press, Baltimore (MD), London, 2001.

\bibitem{MR538060}
P.~Seymour.
\newblock Sums of circuits.
\newblock In {\em Graph theory and related topics ({P}roc. {C}onf., {U}niv.
  {W}aterloo, {W}aterloo, {O}nt., 1977)}, pages 341--355. Academic Press, New
  York, 1979.

\bibitem{BAZ:4875124}
G.~Szekeres.
\newblock Polyhedral decompositions of cubic graphs.
\newblock {\em Bulletin of the Australian Mathematical Society}, 8(3):367--387,
  1973.

\bibitem{MR1395462}
C.-Q. Zhang.
\newblock Nowhere-zero {$4$}-flows and cycle double covers.
\newblock {\em Discrete Mathematics}, 154(1-3):245--253, 1996.

\end{thebibliography}
\end{document}